\documentclass[11pt]{article}

\usepackage{graphicx}
\usepackage{amssymb}
\usepackage{amsmath}
\usepackage{amsthm}
\usepackage{epstopdf}
\usepackage{multicol}
\usepackage{multirow}
\usepackage[numbers]{natbib}
\usepackage{lscape}
\usepackage[ruled,vlined,linesnumbered]{algorithm2e}
\usepackage{adjustbox}
\usepackage{pgfplots, pgfplotstable}
\usepgfplotslibrary{dateplot}
\usepgfplotslibrary{statistics}
\pgfplotsset{compat=1.9}
\usepackage{filecontents}
\usepackage{subcaption}
\usepackage{url}

\usepackage{listings}
\usepackage{color}

\definecolor{javared}{rgb}{0.6,0,0} 
\definecolor{javagreen}{rgb}{0.25,0.5,0.35} 
\definecolor{javapurple}{rgb}{0.5,0,0.35} 
\definecolor{javadocblue}{rgb}{0.25,0.35,0.75} 

\definecolor{gray}{rgb}{0.4,0.4,0.4}
\definecolor{darkblue}{rgb}{0.0,0.0,0.6}
\definecolor{cyan}{rgb}{0.0,0.6,0.6}
 
\lstdefinelanguage{XML}
{
  morestring=[b]",
  morestring=[s]{>}{<},
  morecomment=[s]{<?}{?>},
  stringstyle=\color{black},
  identifierstyle=\color{darkblue},
  keywordstyle=\color{cyan},
  morekeywords={xmlns,version,type}
}

\usepackage{geometry}
\usepackage{authblk}
\usepackage{setspace}

\newtheorem{lemma}{Lemma}

\title{The Dynamic Bowser Routing Problem}
\author[1]{Roberto Rossi\thanks{Corresponding author. Address: 29 Buccleuch place, EH89JS, Edinburgh, UK. Email: roberto.rossi@ed.ac.uk}}
\author[1]{Maurizio Tomasella}
\author[1]{Belen Martin-Barragan}
\author[2]{Tim Embley}
\author[3]{Christopher Walsh}
\author[4]{Matthew Langston}

\affil[1]{Business School, University of Edinburgh, Edinburgh, UK}
\affil[2]{Costain Group Plc., Maidenhead, UK}
\affil[3]{Cenex Advanced Technology Innovation Centre, Loughborough, UK}
\affil[4]{J C Bamford (JCB) Excavators Ltd., Rocester,  UK}
\date{}                                           

\begin{document}
\maketitle


\begin{abstract}
We investigate opportunities offered by telematics and analytics to enable better informed, and more integrated, collaborative management decisions on construction sites. We focus on efficient refuelling of assets across construction sites. More specifically, we develop decision support models that, by leveraging data supplied by different assets, schedule refuelling operations by minimising the distance travelled by the bowser truck as well as fuel shortages. Motivated by a practical case study elicited in the context of a project we recently conducted at Crossrail, we introduce the Dynamic Bowser Routing Problem. In this problem the decision maker aims to dynamically refuel, by dispatching a bowser truck, a set of assets which consume fuel and whose location changes over time; the goal is to ensure that assets do not run out of fuel and that the bowser covers the minimum possible distance. We investigate deterministic and stochastic variants of this problem and introduce effective and scalable mathematical programming models to tackle these cases. We demonstrate the effectiveness of our approaches in the context of an extensive computational study designed around data collected on site as well as supplied by our project partners.\\

\end{abstract}

\section{Introduction}

The UK National Infrastructure Plan comprises a pipeline of public investment in infrastructure worth over \pounds100 billion between 2016 and 2020 \citep{HMTreasury_a} and has clear aspirations for low-carbon solutions \citep{HMGovernment,HMTreasury_b}. Unfortunately, the fragmented nature of construction logistics represents a challenge to these aspirations. Large firms, which are able to secure 80\% of the contracts in the sector, represent only 5\% of the UK construction market players; the remaining 95\% of the industry is composed by SMEs that act as subcontractors for the aforementioned large players. This fragmentation hinders communication among partners. To address this issue, data sharing practices should be fostered within the construction supply chain with the aim of ensuring better use of existing technology --- e.g. real time locating systems embedded in heavy good vehicles --- and of improving operational efficiency. 

In the context of a public-private partnership funded by Innovate UK, our team investigated opportunities offered by telematics and analytics to enable better informed, and more integrated, collaborative management decisions in construction sites. 
\begin{figure}[h!]
\centering
\includegraphics[width=1\columnwidth]{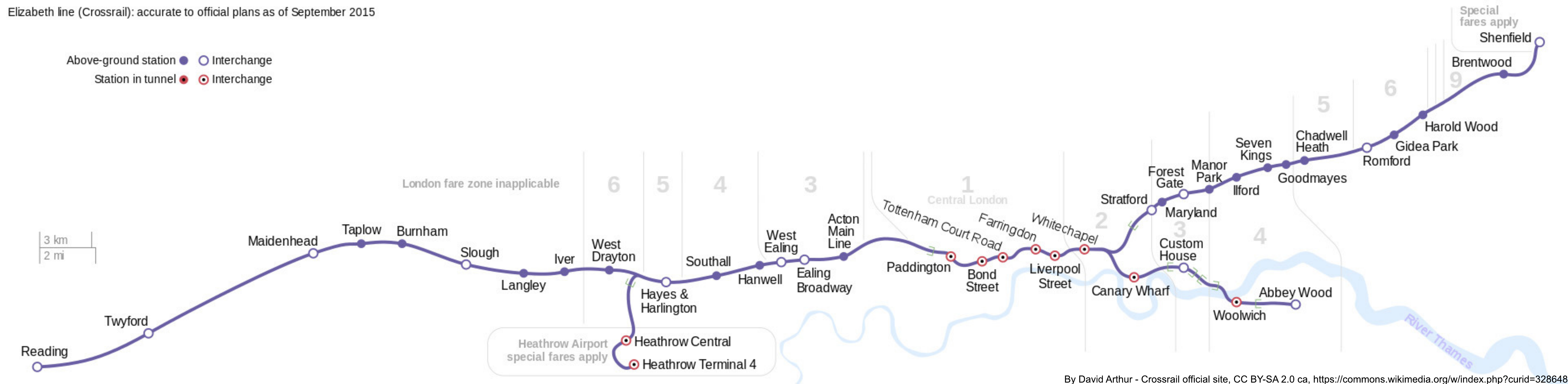}
\caption{The Crossrail project}
\label{fig:crossrail}
\end{figure}
We considered a selection of Crossrail costruction sites, a strategic UK infrastructure project for which Costain is one of the main contractors. Crossrail will deliver a new 118km high frequency, high capacity railway for London and the South East (Fig. \ref{fig:crossrail}). Our team mapped current construction processes and elicited barriers to the fully integrated, low-carbon construction supply chain. Building upon this analysis, we developed a set of solutions and related enabling business models. In this work we describe one of these solutions. 

We focus on efficient refuelling of assets across construction sites. More specifically, we develop decision support models that, by leveraging data supplied by different assets, schedule asset refuelling operations by minimising the distance travelled by the bowser truck as well as fuel shortages. Surprisingly, as we will illustrate in Section \ref{sec:related_works}, there are only few comparable studies in the literature, none of which fully addresses this challenge. Our contributions are the following:
\begin{itemize}
\item motivated by a practical case study elicited in the context of our project, we introduce the Dynamic Bowser Routing Problem (DBRP);
\item we develop mathematical programming models to tackle deterministic as well as stochastic variants of the problem and we augment these models with valid inequalities that enhance computational performances;
\item we introduce a comprehensive test bed built upon real world scenarios and data observed in the context of our experience at Crossrail sites and we carry out a thorough computational study based on it; our mathematical programming models scale well and can tackle instances of realistic size in reasonable time;
\item for the stochastic variant of the problem, we contrast results obtained via our mathematical programming heuristic against the optimal policy obtained via stochastic dynamic programming; our analysis shows that our approximation is effective.
\end{itemize}

The rest of this work is organised as follows. 
In Section \ref{sec:telemetry} we survey telemetry systems and related datasets that have been used in our study. 
In Section \ref{sec:problem_det} we introduce the DBRP and an associated mathematical programming model. 
In Section \ref{sec:problem_sto_sdp} we discuss the Stochastic Bowser Routing Problem (SBRP), which generalises the DBRP to deal with stochastic factors, such as asset fuel consumption and location on site; we discuss a mathematical programming heuristic for the case in which asset fuel consumption is stochastic.
In Section \ref{sec:computational} we present our computational study investigating effectiveness and scalability of our models.
In Section \ref{sec:related_works} we survey related works in the literature. 
Finally, in Section \ref{sec:con} we draw conclusions.

\section{Telemetry at Crossrail construction sites}\label{sec:telemetry}

Modern construction machines, such as excavators, bowser trucks, but also power generators and pumps, feature a plethora of sensors including but not limited to GPS location, fuel level, and engine status. The word {\em telemetry} refers to the transmission of measurements collected by these sensors from the equipment to the point of storage/consumption of the data \citep{citeulike:14100831,1987STIN...8913455.}. By building upon telemetry, {\em telematics}  brings together sensor technologies, telecommunication, and computer science to monitor and control remote objects \citep{nora1978informatisation}. There is a close connection between the fields of telematics and analytics \citep{citeulike:14100832}. Telematics systems heavily rely on visualisation technologies (descriptive analytics), on predictive algorithms (predictive analytics), and on optimisation models (prescriptive analytics) to support and automate decision making.

The introduction of sector-wide standards such as the Association of Equipment Manufacturers Professionals (AEMP) Telematics standard \citep{AEMP2010} has made it possible to collect and integrate data from a wide range of sources located across a single or multiple construction sites. By relying on this standard, live as well as historical data for each piece of equipment can be obtained from Application Programming Interfaces (APIs) that rely on standard HTTP protocols. This opens up a wide range of opportunities to deploy analytics across the construction site. A sample of telemetry data in XML format that can be obtained from assets by leveraging the AEMP standard v1.2 is shown in Appendix I. Sampling rate varies from one manufacturer to another, for instance the JCB LiveLink\textsuperscript{TM} system remotely samples data from assets every five minutes, while Komatsu Komtrax\textsuperscript{TM} features a coarser sampling rate of thirty minutes. 

In this work we exploit two indicators that can be currently monitored via existing telemetry systems: {\em asset (GPS) location} and {\em fuel consumption}: Fig. \ref{fig:odometer} illustrates location of a JCB 540-170 telehandler at the Connaught bridge Crossrail site in London between the 15th and the 19th of March 2016; Fig. \ref{fig:fuel} illustrates cumulative fuel consumption of three JCB 540-170 telehandlers deployed on various Crossrail sites between February and March 2016. 

\begin{figure}
\begin{minipage}{0.45\columnwidth}
\centering
\includegraphics[width=\columnwidth]{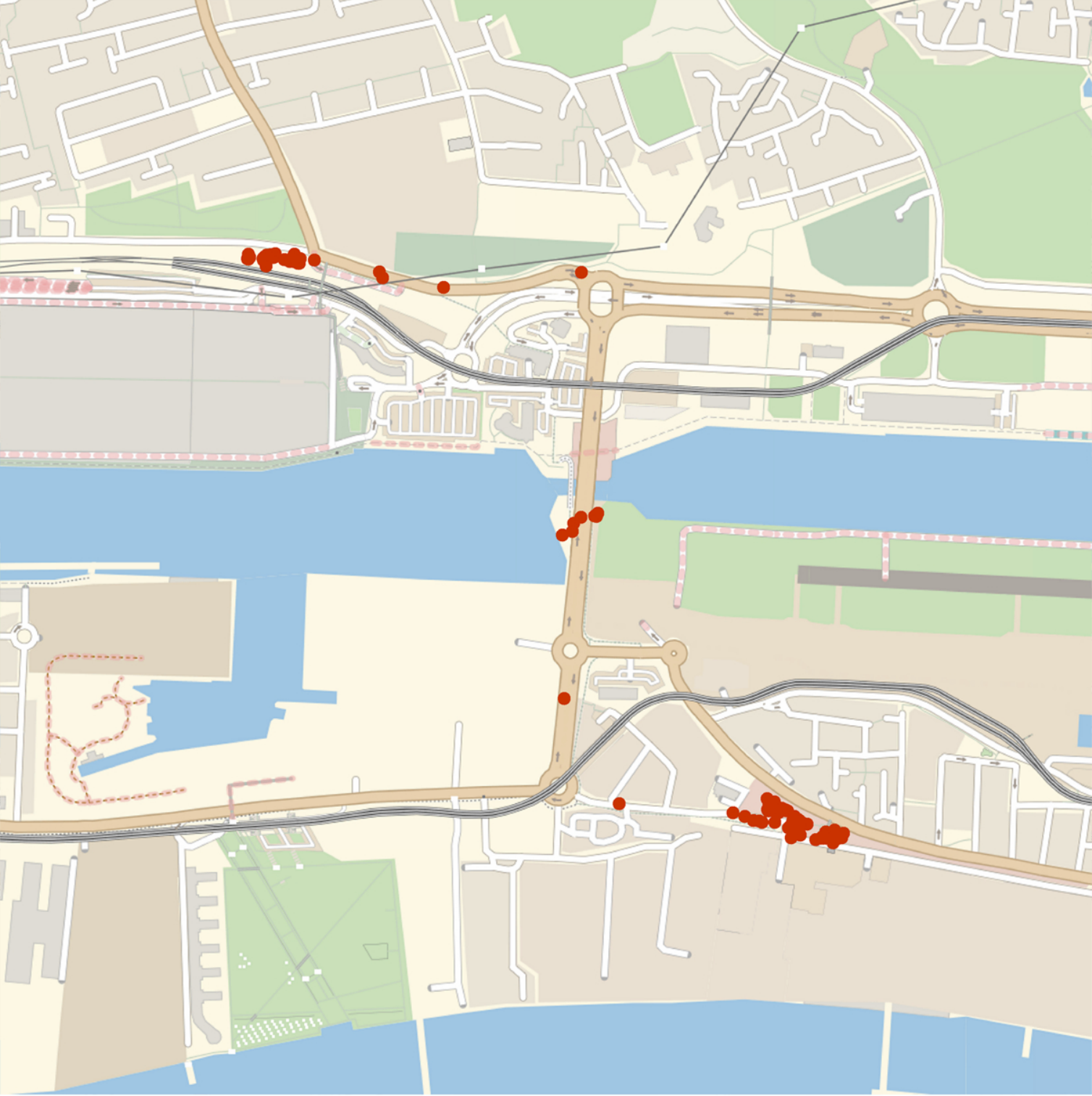}
\caption{GPS location of a JCB 540-170 telehandler at the Connaught bridge Crossrail site in London between the 15th and the 19th of March 2016}
\label{fig:odometer}
\end{minipage}
\hspace{1em}
\begin{minipage}{0.45\columnwidth}
\centering
\begin{tikzpicture}[scale=0.95]
  \begin{axis}[
  legend style={at={(0.03,0.8)},anchor=west},
  date coordinates in=x,
  date ZERO=2016-02-22,
  xticklabel=\month/\day,
  xticklabel style={rotate=45},
  ymin=0,
  ymax=1000]
   \addplot coordinates {
(2016-02-22,0)
(2016-02-23,0)
(2016-02-24,8)
(2016-02-25,23)
(2016-02-26,37)
(2016-02-27,42)
(2016-02-28,49)
(2016-02-29,49)
(2016-03-01,62)
(2016-03-02,74)
(2016-03-03,89)
(2016-03-04,97)
(2016-03-05,97)
(2016-03-06,97)
(2016-03-07,97)
(2016-03-08,97)
(2016-03-09,107)
(2016-03-10,138)
(2016-03-11,171)
(2016-03-12,182)
(2016-03-13,182)
(2016-03-14,182)
(2016-03-15,189)
(2016-03-16,219)
(2016-03-17,254)
(2016-03-18,281)
(2016-03-19,294)
(2016-03-20,294)
   };
\addplot coordinates {
(2016-02-22,0.00)
(2016-02-23,1.00)
(2016-02-24,45.00)
(2016-02-25,84.00)
(2016-02-26,131.00)
(2016-02-27,168.00)
(2016-02-28,178.00)
(2016-02-29,181.00)
(2016-03-01,214.00)
(2016-03-02,256.00)
(2016-03-03,307.00)
(2016-03-04,355.00)
(2016-03-05,385.00)
(2016-03-06,397.00)
(2016-03-07,399.00)
(2016-03-08,440.00)
(2016-03-09,487.00)
(2016-03-10,545.00)
(2016-03-11,608.00)
(2016-03-12,662.00)
(2016-03-13,674.00)
(2016-03-14,674.00)
(2016-03-15,720.00)
(2016-03-16,766.00)
(2016-03-17,816.00)
(2016-03-18,873.00)
(2016-03-19,929.00)
(2016-03-20,942.00)
      };
\addplot coordinates {    
(2016-02-22,0.00)
(2016-02-23,0.00)
(2016-02-24,19.00)
(2016-02-25,39.00)
(2016-02-26,65.00)
(2016-02-27,88.00)
(2016-02-28,88.00)
(2016-02-29,88.00)
(2016-03-01,94.00)
(2016-03-02,128.00)
(2016-03-03,155.00)
(2016-03-04,167.00)
(2016-03-05,190.00)
(2016-03-06,190.00)
(2016-03-07,190.00)
(2016-03-08,224.00)
(2016-03-09,246.00)
(2016-03-10,266.00)
(2016-03-11,292.00)
(2016-03-12,317.00)
(2016-03-13,322.00)
(2016-03-14,322.00)
(2016-03-15,347.00)
(2016-03-16,375.00)
(2016-03-17,390.00)
(2016-03-18,390.00)
(2016-03-19,392.00)
(2016-03-20,392.00)      
      };
\legend{Asset 1, Asset 2, Asset 3}   
  \end{axis}
 \end{tikzpicture}
\caption{Fuel consumption (in liters) of three JCB 540-170 telehandlers deployed on Crossrail sites between February and March 2016}
\label{fig:fuel}  
\end{minipage}
\end{figure}

Telemetry data can be used to track past fuel consumption and predict future consumption via predictive analytics techniques. The analysis of past performance generally proceeds from the visualisation of relevant indicators, for instance in the form of heatmaps (Fig. \ref{fig:heatmap_JCB540-170}).
Standard methods such as linear and non-linear regression, possibly taking into account asset location or other relevant predictor variables, can be employed to predict future fuel consumption. The thorough investigation of suitable predictive analytics strategies to determine future asset fuel consumption or location is beyond the scope of this work. However, in the economy of our investigation it is relevant to determine prototype consumption profiles for a range of assets deployed on construction sites, which we can then use in our computational study. 

\begin{figure}
\centering
\includegraphics[width=1\columnwidth]{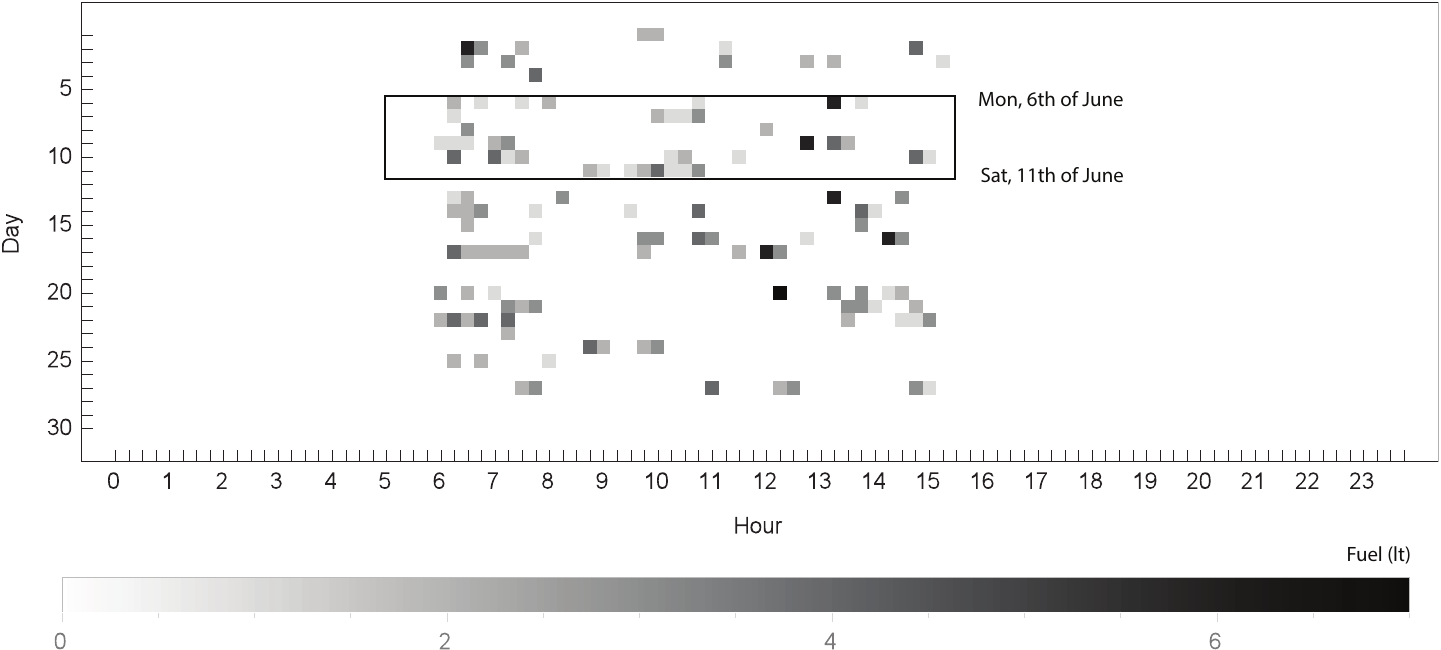}
\caption{Heatmap representing fuel consumption (in litres) for a Crossrail JCB 540-170 telehandler between the 1st of June and the 30th of June 2016; note that time is divided into discrete time periods that last 15 minutes each.}
\label{fig:heatmap_JCB540-170}
\end{figure}

\begin{table}[b]
\centering
\begin{tabular}{l|ll|l}
			&\multicolumn{2}{c|}{Compound Poisson}\\
Asset model	&$\lambda$&jump size distribution&p-value\\
\hline
JCB 540-170	&0.502645&Poisson(0.602257)&0.91558\\
JCB 540-170	&0.774271&Poisson(0.684164)&0.449291\\
JCB 540-170	&0.3731890&Poisson(1.004940)&0.933036\\	
JCB JS130	&1.03892&Poisson(1.01056)&0.460517\\	
JCB JS130	&0.926141&Poisson(0.393873)&0.116692\\	
JCB 86C-1	&0.476964&Poisson(0.960902)&0.778792\\
JCB 531-70	&0.283428&Poisson(0.0516331)&0.516138
\end{tabular}
\caption{Fitted distribution for a selection of JCB assets deployed on Crossrail sites in June 2016; the distribution represent the fuel consumption over a 15 minutes time bucket.}
\label{tab:fuel_consumption}  
\end{table}

By looking at the heatmap in Fig. \ref{fig:heatmap_JCB540-170} it is clear that asset utilisation follows specific patterns; for instance, this asset operates from early morning to mid afternoon and does not generally operate over weekends --- note that different assets feature different profiles, for instance some assets operate 24/7. In this specific instance, if we consider a time window during which the asset is in use --- e.g. the time window from Mon, 6th of June to Sat, 11th of June highlighted in Fig. \ref{fig:heatmap_JCB540-170} --- a compound Poisson distribution with parameter $\lambda=0.502645$ and jump size distribution Poisson($0.602257$) appears to provide a good fit for the consumption pattern observed during independent time buckets of 15 minutes. Our empirical study revealed that a compound Poisson distribution with Poisson jump size distribution generally provides a good fit for asset consumption data we collected in June 2016; by exploiting a month worth of fuel consumption data collected from different types of assets, including a variety of telehandlers and excavators, we excluded periods of inactivity (i.e. nights and weekend) and we fitted distribution parameters using a maximum likelihood approach; the resulting distributions are shown in Table \ref{tab:fuel_consumption}, the respective p-values are also reported in the table. Further details on the analysis carried out, including heatmaps for assets analysed, are provided in Appendix II.


In addition to analyzing asset fuel consumption distributions, we also analysed the behavior of a 7.5 tonne bowser truck fitting a 950 litres fuel tank that refuels assets across a number of Crossrail construction sites. We installed a logger on the truck and tracked its movements over a  working week from Mon, 20th of June to Fri, 24 of June. The daily distance covered by the bowser truck and the number of journeys are reported in Fig. \ref{fig:bowser_c610_journeys}; Fig. \ref{fig:bowser_journeys_percentage} shows distance covered at different hours.

\begin{figure}
\begin{minipage}[!t]{0.45\columnwidth}
\centering
\begin{tikzpicture}[scale=0.9]
  \begin{axis}[
  legend style={at={(0.03,0.8)},anchor=west},
  date coordinates in=x,
  date ZERO=2016-02-22,
  xticklabel=\month/\day,
  xticklabel style={rotate=45},
  ymin=0,
  ymax=170,
  ybar]
   \addplot coordinates {
(2016-06-20,1.37697)
(2016-06-21,99.8742784)
(2016-06-22,98.67171)
(2016-06-23,33.616935)
(2016-06-24,119.6985102)
   };
\addplot coordinates {
(2016-06-20,4)
(2016-06-21,20)
(2016-06-22,12)
(2016-06-23,18)
(2016-06-24,17)
      };   
\legend{Daily distance (km), Number of journeys}   
  \end{axis}
 \end{tikzpicture}
\caption{Daily distance covered by the bowser truck and number of journeys completed on a specific date}
\label{fig:bowser_c610_journeys}  
\end{minipage}
\hspace{2em}
\begin{minipage}[!t]{0.45\columnwidth}
\centering
\begin{tikzpicture}[scale=0.95]
  \begin{axis}[
  yticklabel=\pgfmathprintnumber{\tick}\,\%,
  ymin=0,
  ymax=30,
  ybar]
   \addplot coordinates {
(7,1.6)
(8,1.48)
(9,4.48)
(10,24.85)
(11,3.37)
(12,25.18)
(13,17.36)
(14,11.39)
(15,4.32)
(16,2.34)
(17,3.63)
   };
  \end{axis}
   \end{tikzpicture}
  \caption{Percentage distance covered by the bowser at a given hour of the day}
\label{fig:bowser_journeys_percentage}  
\end{minipage}
\end{figure}

\begin{figure}
\centering
\includegraphics[width=1\columnwidth]{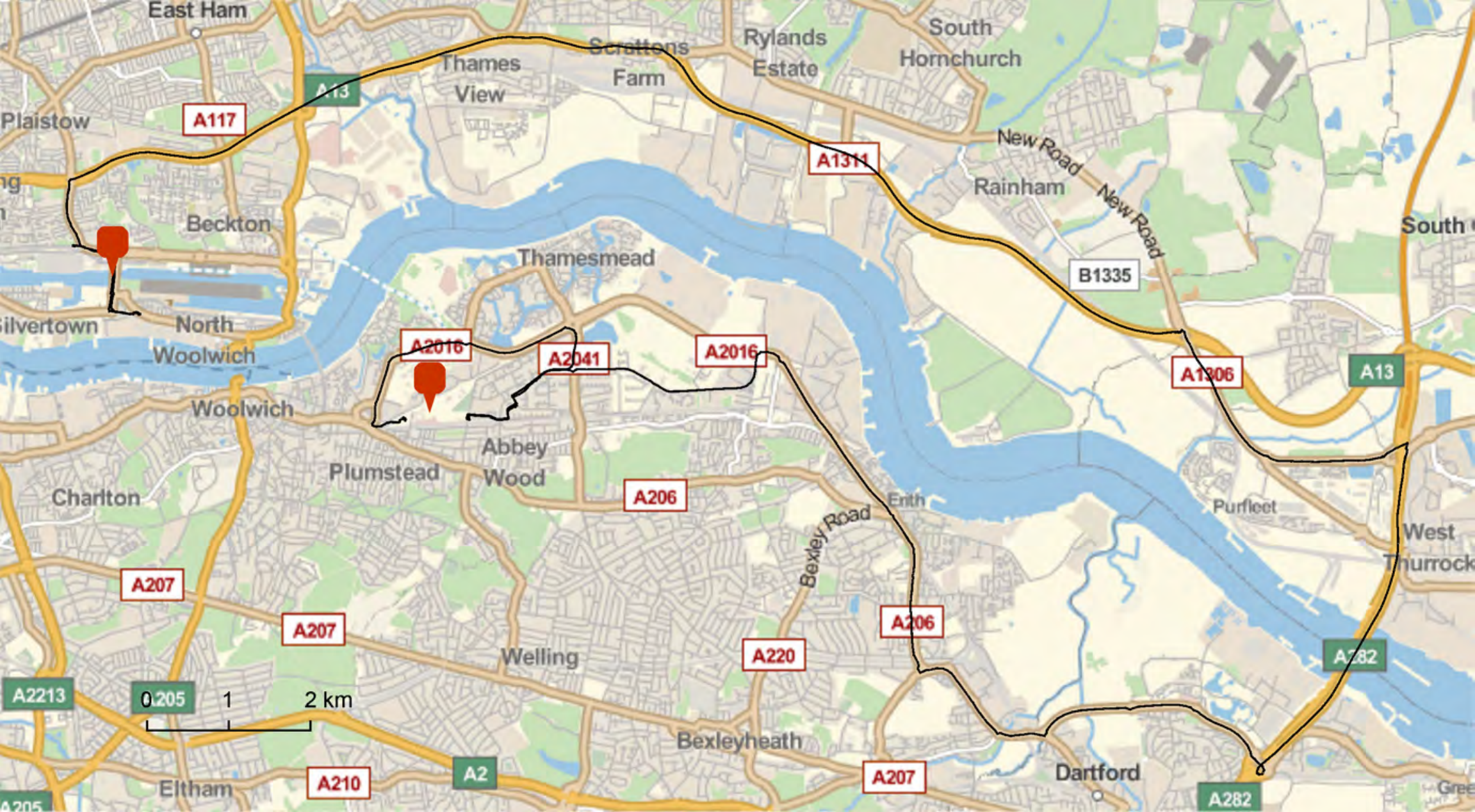}
\caption{A bowser journey from the Connaught bridge site in North Woolwich to the Plumstead site of Crossrail on Tue, 21th of June; the bowser covered 47.6km in 1h and 31 minutes.}
\label{fig:bowser_long_journey}
\end{figure}

\begin{figure}
\centering
\includegraphics[width=1\columnwidth]{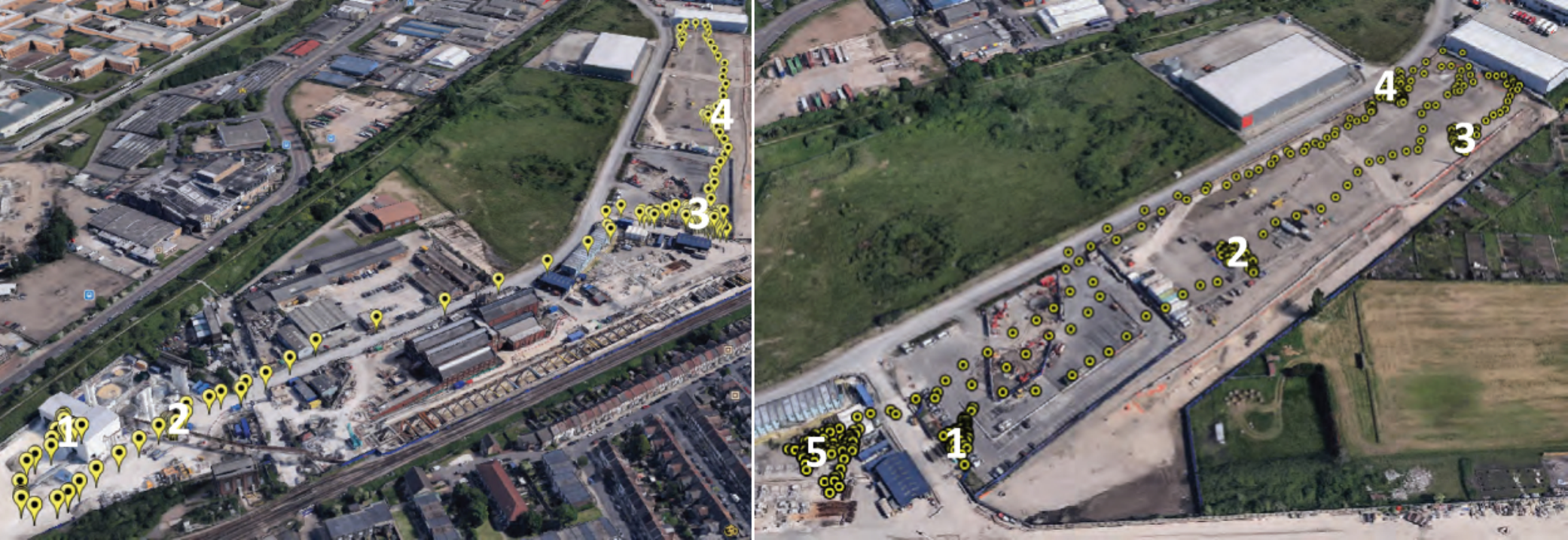}
\caption{Two short on-site journeys at Plumstead site; the first journey (on the left hand side) covered 1.3km in 44 minutes and comprises four stops of less than five minutes each; the second journey (on the right hand side) covered 1.7km in 37 minutes and comprises five stops of less than five minutes each.}
\label{fig:bowser_short_journeys}
\end{figure}

A typical day for the bowser truck comprises two long (approx 40km) journeys across construction sites, and about 20 short journeys carried out within specific construction sites. An example of a long journey between the Connaught bridge site and the Plumstead site site of Crossrail is shown in Fig. \ref{fig:bowser_long_journey}; examples of short on-site journeys at Plumstead site are given in Fig. \ref{fig:bowser_short_journeys}. 

In the rest of this work we introduce optimisation models that rely upon telemetry data to schedule asset refuelling operations across construction sites.

\section{The Dynamic Bowser Routing Problem}\label{sec:problem_det}

Motivated by the discussion in the previous section, we introduce the Dynamic Bowser Routing Problem (DBRP). A comprehensive table of symbols used in the rest of this work is introduced in Appendix III. We consider a construction site with $A$ assets (e.g. generators, telehandlers, excavators, etc) all powered by a single type of fuel (e.g. diesel). The construction site map is given. This map takes the form of a directed graph $\langle V, E \rangle$, where $V$ is the set of nodes and $E$ is the set of arcs. We assume this graph to be connected, but not necessarily fully connected. Nodes in $V$ represents relevant locations across the construction site, while arcs in $E$ represent travel distances between locations. Essentially, this graph can be seen as a network that represents adjacent accessible locations (i.e. nodes), and their respective direct travel distances (i.e. arcs), across the construction site. This form of map representation is common in GPS navigation systems and can be easily obtained via GPS traces.\footnote{see e.g. \url{https://www.openstreetmap.org/}}
\begin{figure}
\centering
\includegraphics[width=0.5\columnwidth]{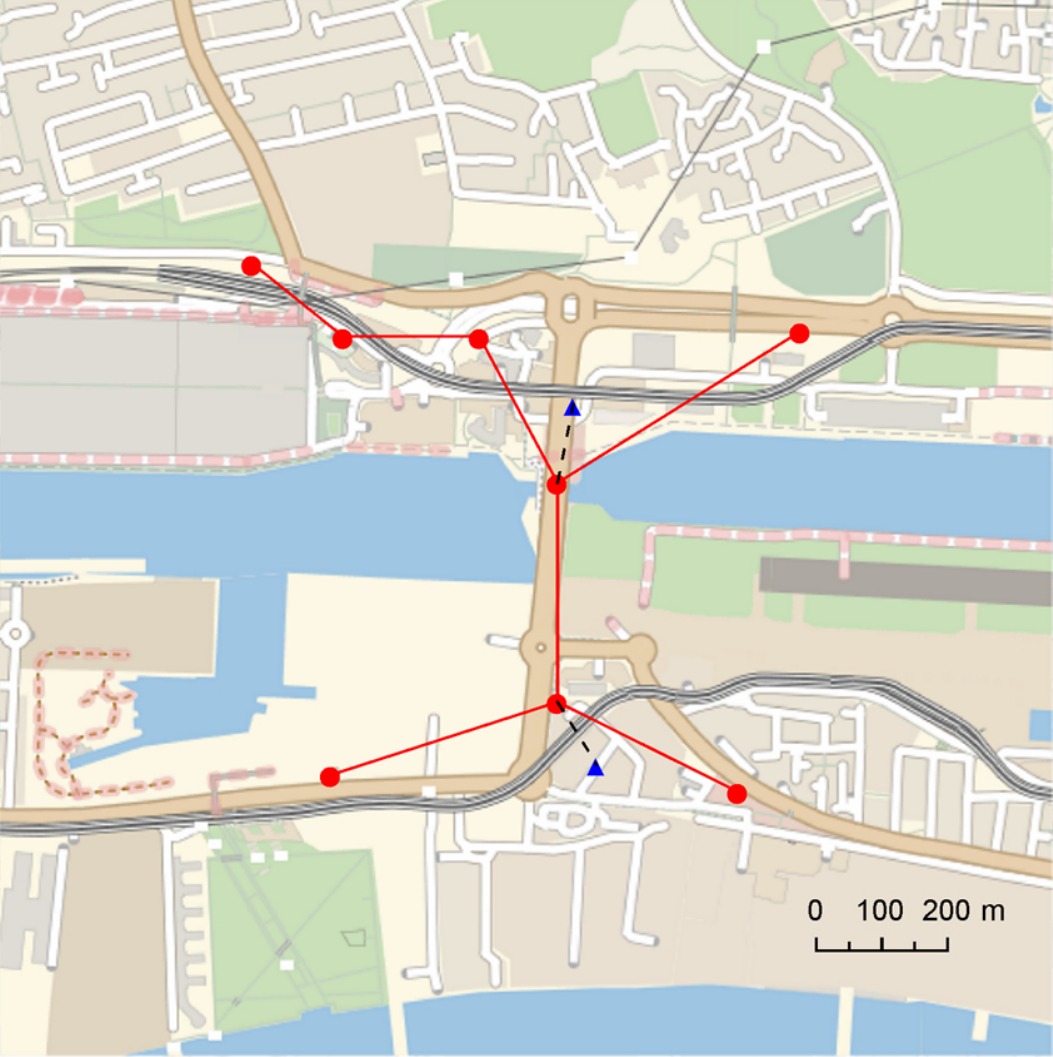}
\caption{Sample site network for the Connaught bridge Crossrail site in London; triangles represent assets.}
\label{fig:overlay}
\end{figure}
It should be also noted that this setup is not limited to a single construction site and the graph may as well represent multiple interconnected sites.

We consider a discrete planning horizon that comprises $T$ periods. At each point in time, an asset $a=1,\ldots,A$ can be found in one and only one node $v\in V$; in practice, on the basis of its GPS location an asset will be associated to the nearest node in the site network (Fig. \ref{fig:overlay}). We assume that asset location $l^a_t\in V$ at each time period $t\in T$ is known with certainty. Each asset $a$ features a fuel tank with capacity $c_a$; fuel consumption $f^a_t\geq0$ for an asset at a given time period $t\in T$ is given and known with certainty. We will relax the assumptions of certainty in Section \ref{sec:problem_sto}. 

Each construction site features a unmovable site cistern where an infinite amount of fuel is assumed to be available; in a multi-site setting we assume, without loss of generality, a single cistern is available for all sites considered. There is a single bowser truck that can be used to refuel assets. The bowser features a tank with capacity $c_b$; when the bowser tank is empty, the bowser must return to the site cistern to refill its tank. We do not model explicitly bowser fuel consumption, since it is unlikely the bowser will run out of fuel in between two visits to the site cistern: the daily distance covered by the bowser (Fig. \ref{fig:bowser_c610_journeys}) comfortably remains within the range a small truck can cover with a full tank. 

We model movements so that the bowser can only move from a node to an adjacent one within a single time period. We assume that refuelling of an asset takes a negligible time in relation to the size of time periods, and that refuelling can be performed if, at a given time period, both the bowser and the asset are located at the same node. 

Our time modeling strategy resembles a ``Large Bucket'' strategy, as found in lot-sizing \citep{citeulike:14071902}. The rationale behind our choice of modeling macro (i.e. 10 minutes to 30 minutes) rather than micro (i.e. real time) periods is related to the fact that, as discussed in Section \ref{sec:telemetry}, asset refuelling operations generally require less than five minutes per asset and therefore fit within a period. Furthermore, the sampling rate of existing telemetry systems is quite low: between 5 minutes (JCB) and 30 minutes (Komatsu) between two readings. 

Even in a deterministic setting, it is unrealistic to require that no asset stocks out, as the problem may admit no solution; we therefore choose to allow asset fuel stock outs. If an asset stocks out of fuel, we enforce a penalty cost $p$ per litre of fuel short at the end of a given period. We assume operations do not stop as a consequence of a fuel shortage; a premium price will be paid to ensure work continuity, e.g. cost of expediting fuel, cost of allocating a spare asset to the task. Our asset therefore operate in a ``lost sales'' setting, i.e. when an asset is out of fuel lost work and associated fuel consumption are not backlogged from one period to the next.

To summarize, the order of events within a given period $t$ in the planning horizon is as follows. At the beginning of a period a bowser replenishment takes place if the bowser is at the cistern node and needs to be replenished; this replenishment is instantaneous. Immediately after, all assets that happen to be at the same node as the bowser are replenished according to the given refuelling plan. Assets then start operating and consume fuel according to the given consumption $f^a_t$. If an asset runs out of fuel, a premium price of $p$ per litre of fuel short is paid to ensure business continuity. At the end of the period all assets, including the bowser, move instantaneously to their next location.

\subsection{A Mixed-Integer Linear Programming model}

We introduce a mixed-integer linear programming (MILP) formulation for the DBRP. Our formulation features five sets of decision variables 
\begin{itemize}
\item $V^i_t$, a binary variable set to one iif, at time $t$, the bowser is at node $i$; 
\item $T_t^{i,j}$ a binary variable set to one iif the bowser transits from node $i$ to node $j$ at the end of period $t$;
\item $Q^a_t$, the nonnegative quantity of fuel delivered to asset $a$ at time $t$; 
\item $B_t$, the nonnegative quantity of fuel transferred from the cistern to the bowser at time $t$; 
\item $S^a_t$, the nonnegative fuel shortage for asset $a$ at time $t$.
\end{itemize}

\begin{figure}
\centering
\fbox{\parbox{14cm}{
\begin{align}
\min ~~ \sum_{t=2}^T \sum_{i=1}^N \sum_{j=1}^N T_{t-1}^{i,j}d_{i,j}+p\sum_{t=1}^T\sum_{a=1}^A S^a_t\label{cons:milp_obj}
\end{align}

{\bf Subject to}
\begin{align}
&B_t \leq V^1_tc_b														&t=1,\ldots,T\label{cons:milp1}\\
&s_b + \sum_{k=1}^t B_k - \sum_{k=1}^{t-1} \sum_{a=1}^A Q^a_k \leq c_b			&t=1,\ldots,T\label{cons:milp2}\\
&s_b + \sum_{k=1}^t B_k - \sum_{k=1}^t \sum_{a=1}^A Q^a_k \geq 0				&t=1,\ldots,T\label{cons:milp3}\\
&\sum_{i=1}^N V^i_t = 1													&t=1,\ldots,T\label{cons:milp4}\\
&\delta_{i,j}\geq V^i_{t-1} + V^j_t -1											&t=2,\ldots,T;~i,j=1,\ldots,N\label{cons:milp5}\\
&\sum_{j=1}^N T_{t-1}^{i,j} = V^i_{t-1}										&t=2,\ldots,T;~i=1,\ldots,N\label{cons:milp6}\\
&T_{t-1}^{i,j}\geq V^i_{t-1}+V^j_{t}-1											&t=2,\ldots,T;~i,j=1,\ldots,N\label{cons:milp7}\\
&T_{t-1}^{i,j}\leq V^i_{t-1}													&t=2,\ldots,T;~i,j=1,\ldots,N\label{cons:milp8}\\
&T_{t-1}^{i,j}\leq V^j_{t}													&t=2,\ldots,T;~i,j=1,\ldots,N\label{cons:milp9}\\
&s_a+\sum_{k=1}^t Q^a_k+\sum_{k=1}^{t-1}S_k^a-\sum_{k=1}^t f^a_k\geq - S_t^a		&t=1,\ldots,T;~a=1,\ldots,A\label{cons:milp10}\\
&s_a+\sum_{k=1}^t Q^a_k+\sum_{k=1}^{t-1}S_k^a-\sum_{k=1}^{t-1} f^a_k\leq c_a		&t=1,\ldots,T;~a=1,\ldots,A\label{cons:milp11}\\
&Q^a_t\leq c_a\sum_{i=1}^N V^i_{t}l^a_{t,i}									&t=1,\ldots,T;~a=1,\ldots,A\label{cons:milp12}\\
\begin{split}
&T_t^{i,j},~V^i_t\in\{0,1\}\\
&Q^a_t,~B_t\geq 0\\
& 0\leq S^a_t\leq f^a_t
\end{split}&\label{cons:milp13}
\end{align}
}}
\caption{An MILP model for the Dynamic Bowser Routing Problem}
\label{model:milp}
\end{figure}

The MILP model is presented in Fig. \ref{model:milp}. The objective function (\ref{cons:milp_obj}) minimises the total distance covered by the bowser plus the penalty cost associated with fuel shortages; note that both these values can be expressed in the same unit of measure, e.g. monetary units.
Constraint (\ref{cons:milp1}) ensures that fuel cannot be transferred from the cistern to the bowser unless the bowser is at node 1. Constraint (\ref{cons:milp2}) enforces bowser capacity. Constraint (\ref{cons:milp3}) introduce inventory conservation constraints for the bowser. Constraint (\ref{cons:milp4}) captures the fact that at each point in time the bowser must be somewhere in the network. Constraint (\ref{cons:milp5}) states that the bowser can transit from node $i$ to node $j$ only if these two are connected. Constraint (\ref{cons:milp6}) establishes that if the bowser is at node $i$ at time t, it must transit somewhere. Constraints (\ref{cons:milp7}, \ref{cons:milp8}, \ref{cons:milp9}) are channeling constraints linking variables $T_t^{i,j}$ and variables $V^i_t$. Asset refuelling and inventory conservation constraints are expressed via constraints (\ref{cons:milp10}) and (\ref{cons:milp11}). The fact that an asset can be refuelled only if both machine and boswer are at the same node is captured by constraint (\ref{cons:milp12}). Finally, constraints \ref{cons:milp13} capture decision variable domains.

\subsection{Valid Inequalities}

We next introduce three sets of valid inequalities \citep{citeulike:14541797} inspired by the discussion in \citep{citeulike:14131453}.

The {\em first set of inequalities} captures the idea that the amount of fuel that can be delivered to assets is constrained by the number of visits the bowser pays to each asset as well as by the asset as well as bowser capacities. 

\begin{lemma}
For $t=1,\ldots,T$ and $a=1,\ldots,A$
\begin{equation}
\sum_{k=1}^t\sum_{i=1}^N V^i_{k}l^a_{k,i} \geq \left(\sum_{k=1}^t f^a_k - s_a - \sum_{k=1}^t S_k^a\right)/\min(c_a, c_b)
\end{equation}
\end{lemma}
\begin{proof}
$Q=\sum_{k=1}^t f^a_k - s_a - \sum_{k=1}^t S_k^a$ is the net amount of fuel that has been delivered to the asset by period $t$; the minimum value between asset $a$ tank capacity $c_a$ and the bowser capacity $c_b$ can be used to determine the minimum number of visits required to deliver $Q$, the total number of visits paid must then be greater or equal to this value.
\end{proof}

The {\em second set of inequalities} makes sure that if an asset $a$ has not been visited by the bowser in the time span $i,\ldots,j$, then the sum of asset $a$ refuelling quantities $Q^a_k$ in this time span must be set to zero.
\begin{lemma}
For $i,j=1,\ldots,T$ and $a=1,\ldots,A$
\begin{equation}
M\sum_{n=1}^N \sum_{k=i}^j V^n_{k} l^a_{k,n} \geq \sum_{k=i}^j Q^a_k
\end{equation}
where $M=\sum_{k=1}^T f^a_k$.
\end{lemma}
\begin{proof}
These constraints immediately follow from the description above; $M=\sum_{k=1}^T f^a_k$, since the total amount of fuel delivered to an asset should not exceed its total consumption up to time $T$.
\end{proof}

The {\em third set of inequalities}  captures the fact that, if the bowser does not pay a visit to asset $a$ in the time span $j,\ldots,t$, then the net fuel level at period $t$ is independent of $Q^a_j,\ldots,Q^a_t$.
\begin{lemma}
For $a=1,\ldots,A$, $t=1,\ldots,T$ and $j=1,\ldots,t$
\begin{equation}
s_a+
\sum_{k=1}^{j-1} Q^a_k + 
\sum_{k=1}^{t} S_k^a -
\sum_{k=1}^{t} f^a_k \geq
-M\sum_{n=1}^N \sum_{k=j}^t V^n_{k} l^a_{k,n}
\end{equation}
where $M=\sum_{k=1}^t f^a_k$.
\end{lemma}
\begin{proof}
These constraints immediately follow from the description above; $M=\sum_{k=1}^t f^a_k$; since a right hand side value of $\sum_{k=1}^t f^a_k$ makes the constraint implied.
\end{proof}

\subsection{Numerical example}\label{sec:numerical_example_1}

We now introduce a running example to support our discussion. We consider a planning horizon of $T=10$ periods and a single site network with $N=10$ nodes. The distance matrix is shown in Table \ref{tab:distance_matrix}. 
\begin{table}[h!]
\centering
\begin{tabular}{rr|rrrrrrrrrr}
\multicolumn{2}{c}{}&\multicolumn{10}{c}{$j$}\\
&&1&2&3&4&5&6&7&8&9&10\\
\hline
\multirow{10}{*}{$i$}&1& & 96& & & 107& & & & & \\
&2&121& & & & & & & & & \\
&3& & 92& & 103& 103& & & 92& & 77\\
&4& & 90& & & & & & & & 91\\
&5& & & & & & 102& & & 126& \\
&6& & & 72& & 139& & 89& & & \\
&7& & & & & & 80& & 83& & \\
&8& & 119& & & & & & & 90& 91\\
&9&83& & & & & & & & & \\
&10& & & & & 79& & & & & \\
\end{tabular}
\caption{Distance matrix representing the distances $d_{i,j}$ for the numerical example; an empty cell denotes a forbidden transit}
\label{tab:distance_matrix}
\end{table}
There are $A=3$ assets with tank capacity $c_a=20$ and initial tank level $s_a=10$ for all $a=1,\ldots,A$. The location of each asset at each time period is illustrated in Table \ref{tab:asset_location}.
\begin{table}[h!]
\centering
\begin{tabular}{r|rrrrrrrrrr}
\multicolumn{1}{c}{}&\multicolumn{10}{c}{Period}\\
		& 1 & 2 & 3 & 4 & 5 & 6 & 7 & 8 & 9 & 10\\
\hline
Asset 1	& 5 & 10 & 1 & 1 & 4 	& 4 & 2	& 6 & 6 & 6\\
Asset 2	& 6 & 1 & 9 & 3 & 2 		& 3 & 7	& 7 & 10 & 5\\
Asset 3	& 10 & 5 & 7 & 5 & 10 	& 2 & 10 	& 10 & 6 & 3
\end{tabular}
\caption{Location (i.e. node index in the site graph) of each asset at each time period.}
\label{tab:asset_location}
\end{table}
Fuel consumption of each asset at each time period is illustrated in Table \ref{tab:fuel_consumption}; the penalty cost per litre of fuel short is $p=100$.
\begin{table}[h!]
\centering
\begin{tabular}{r|rrrrrrrrrr}
\multicolumn{1}{c}{}&\multicolumn{10}{c}{Period}\\
		& 1 & 2 & 3 & 4 & 5 & 6 & 7 & 8 & 9 & 10\\
\hline
Asset 1&4&4&2&1&3&1&4&4&3&3\\
Asset 2&4&2&3&4&3&1&4&2&4&4\\
Asset 3&2&4&1&2&2&4&1&1&2&2
\end{tabular}
\caption{Fuel consumption (in liters) of each asset at each time period.}
\label{tab:fuel_consumption}
\end{table}
Finally, the bowser capacity is $c_b=300$ and the initial bowser level is $s_b=10$. An IBM ILOG OPL implementation of this example is provided in our Electronic Addendum EA1.\footnote{Our Electronic Addendum is available on GitHub \url{https://github.com/gwr3n/dbrp}.}

We solve the MILP model discussed by using IBM ILOG CPLEX Optimization Studio Version 12.6; the optimal solution, which yields a cost of 494, is found in 1.08 seconds on a 2.2GHz Intel Core i7 Macbook Air fitted with 8Gb of RAM. In this specific instance, this figure represents the total distance covered by the bowser is 494 with no fuel shortages. The optimal bowser routing plan is displayed in Fig. \ref{fig:routing_plan} and in Table \ref{tab:routing_plan_small}; the optimal refuelling plan is shown in Fig. \ref{fig:refueling_plan}. 

\begin{figure}[h!]
\begin{minipage}[!t]{0.45\columnwidth}
\centering
\includegraphics[width=1\columnwidth]{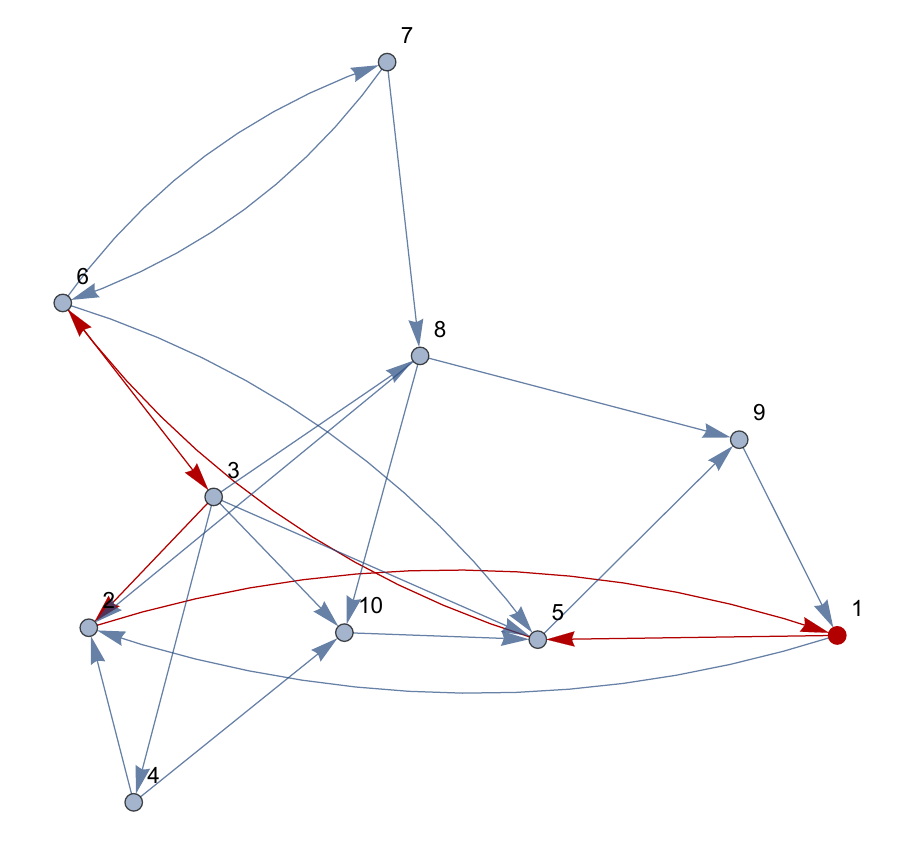}
\caption{Optimal bowser routing plan.}
\label{fig:routing_plan}
\end{minipage}
\hspace{2em}
\begin{minipage}[!t]{0.45\columnwidth}
\centering
\begin{tabular}{llllll}
$t$	&	Transition			\\
\hline
1	&	1$\rightarrow$	1	\\
2	&	1$\rightarrow$	1	\\
3	&	1$\rightarrow$	5	\\
4	&	5$\rightarrow$	6	\\
5	&	6$\rightarrow$	3	\\
6	&	3$\rightarrow$	2	\\
7	&	2$\rightarrow$	1	\\
8	&	1$\rightarrow$	1	\\
9	&	1$\rightarrow$	1			
\end{tabular}
\caption{Optimal bowser routing plan.}
\label{tab:routing_plan_small}
\end{minipage}
\end{figure}
\begin{figure}[h!]
\centering
\includegraphics[width=0.7\columnwidth]{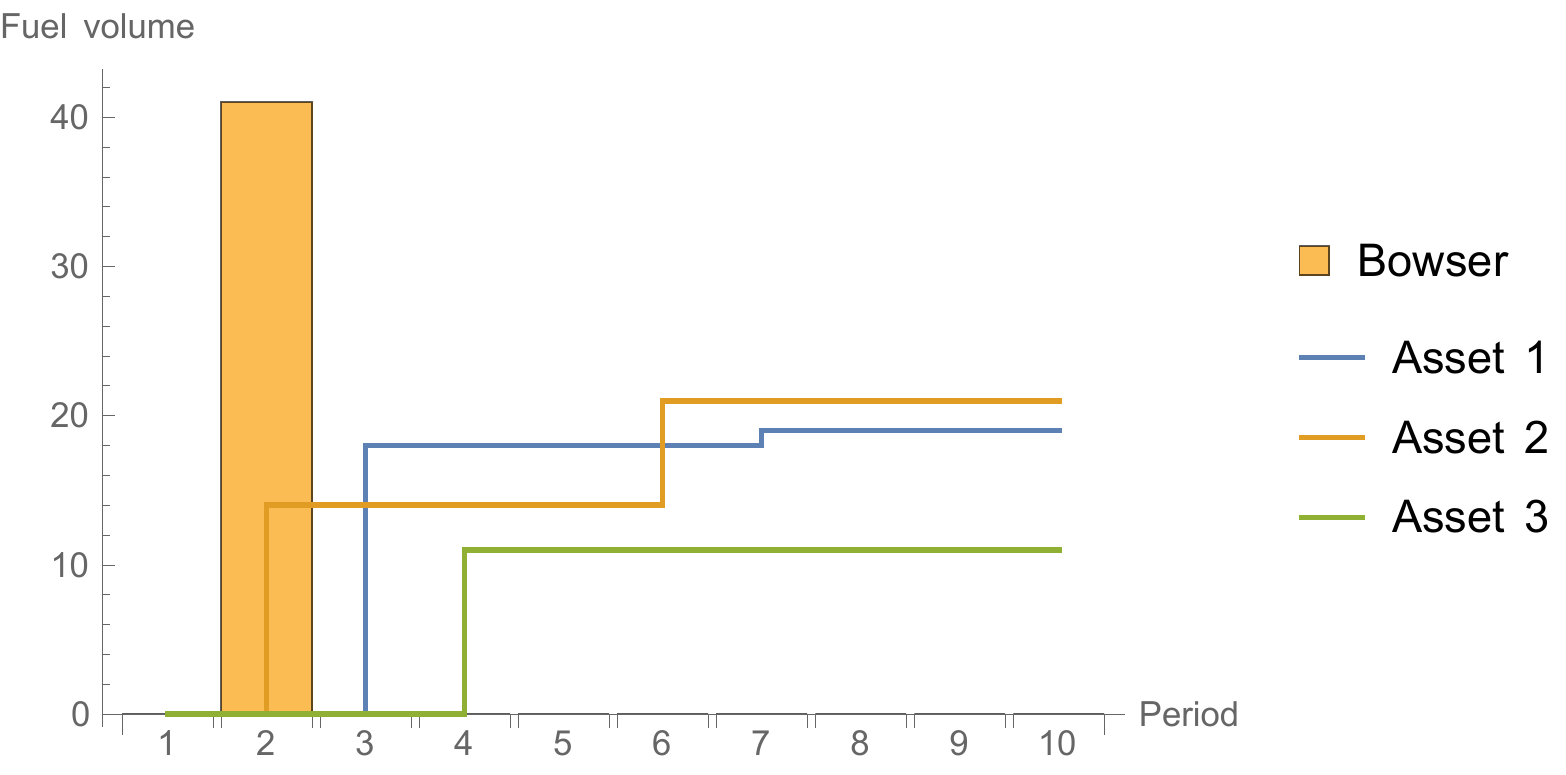}
\caption{Bowser and asset refueling plan.}
\label{fig:refueling_plan}
\end{figure}

\section{The Stochastic Bowser Routing Problem}\label{sec:problem_sto}

In the DBRP asset {\em location} and {\em fuel consumption} throughout the planning horizon are known with certainty. In this section we investigate stochastic variants of the problem; we shall name this problem the Stochastic Bowser Routing Problem (SBRP).

\subsection{A Stochastic Dynamic Programming formulation}\label{sec:problem_sto_sdp}

We now introduce a dynamic programming formulation that can seamlessly capture deterministic as well as stochastic variants of the problem; in other words, this approach makes it possible to relax the assumption of certainty for asset {\em location} or {\em fuel consumption}. This flexibility of course comes at a price; in fact, the curse of dimensionality makes this approach far less scalable than other approaches previously discussed. However, at least for small instances, this formulation can be effectively employed to derive optimal policies and investigate the cost of uncertainty.

Introduced by Bellman in the late Fifties \citep{Bellman:1957} Dynamic Programming is a powerful modelling and solution framework for decision making under uncertainty. To keep our discussion focused, we restrict the discussion to discrete stochastic dynamic programs. To formulate a problem as a stochastic dynamic program the decision maker must define the {\em planning horizon} length, the {\em state space}, the set of {\em feasible actions} for each state of the state space, the {\em transition probability} from one state-action pair to the set of admissible future states and the {\em immediate value} function associated with every given state-action pair. At the core of stochastic dynamic programs we find a ``functional equation'' --- also known as Bellman's equation. In its most general form, a functional equation $v_t(s)$ can be expressed as
\begin{equation}\label{eq:bellman_equation}
v_t(s)=\min_{a\in A_s} \left\{ c^a_s + \sum_{s'} p^a_{s,s'} v_{t+1}(s') \right\}, 
\end{equation}
where $A_s$ denote the set of all feasible actions in state $s$; $c^a_s$ denotes the immediate cost incurred for the state-action pair $\langle s,a\rangle$; and $p^a_{s,s'}$ is the transition probability from state $s$ to state $s'$ when action $a$ is selected. Eq. \ref{eq:bellman_equation} is clearly independent of the problem at hand. Note that, without loss of generality,  the problem is here expressed in cost minimization form. Once all aforementioned problem elements have been defined, one can apply a backward or a forward recursion algorithm to tabulate the functional equation --- a process known in the literature as ``memoization'' \citep{citeulike:2851015} --- and determine an optimal policy.

We next formulate the SBRP as a stochastic dynamic program. 

{\bf Planning horizon:} $T$ periods.

{\bf States:} a state is encoded as a 5-tuple $\langle t, b_{\text{tank}}, b_{\text{loc}}, m_{\text{tank}}, m_{\text{loc}} \rangle$; where $t$ is the period associated with the state, $b_{\text{tank}}$ is the bowser tank level, $b_{\text{loc}}$ is the bowser position in the network, $m_{\text{tank}}$ is an array of $A$ asset tank levels, $m_{\text{loc}}$ is an array of $A$ asset locations.

{\bf Actions:} an action is a 4-tuple $\langle s, b_{\text{ref}}, b_{\text{loc}}', m_{\text{ref}} \rangle$; where $s$ is the state associated with the action, $b_{\text{ref}}$ is the bowser refuelling quantity for the current period, $b_{\text{loc}}'$ is the bowser location in the next period, $m_{\text{ref}}$ is an array of $A$ asset refuelling quantities.  

{\bf Transition probabilities:} the transition probabilities depend on what problem parameters are modelled as random variables. We have developed two variants of the problem: one in which asset movements from period $t$ to period $t+1$ are captured by means of a probability mass function over the nodes of the construction site graph; and another in which an asset fuel consumption in period $t$ is expressed as a generic probability distribution over a discrete support. In both cases, transition probabilities $p^a_{s,s'}$ are immediately obtained from the given distributions.

{\bf Immediate value function:} the immediate value function is given by the distance travelled by the bowser over the arc $(b_{\text{loc}}, \bar{b}_{\text{loc}})$ selected by action $a$, plus the expected penalty cost paid at the end of the period for each unit of fuel short.

Once all problem elements have been defined, the decision maker may tackle the problem numerically by applying established complete approaches such as forward or backward recursion algorithms, or approximate dynamic programming algorithms \citep{NAV:NAV20347}. We implemented the model discussed using an open source general purpose library called \texttt{jsdp}.\footnote{http://gwr3n.github.io/jsdp/}. Our implementation of the DBRP is available in package \texttt{jsdp.app.routing}; this package comprises a deterministic formulation, as well as two stochastic variants of the problem in which {\em asset location} and {\em asset fuel consumption}, respectively, are modelled as random variables. Unfortunately, this approach can only tackle small instances.

In what follows we will present an effective Mixed-Integer Linear Programming heuristic for the case in which asset fuel consumption is random. We leave the investigation of effective heuristics for the random asset location case as future work.

\subsection{A Mixed-Integer Linear Programming heuristic}\label{sec:problem_sto_milp}

We consider an SBRP in which asset $a$ fuel consumption in period $t$ is a random variable $f^a_t$ with known probability distribution over nonnegative support, e.g. Poisson; our aim is to model and solve heuristically this problem.

The SBRP under random fuel consumption is a complex multi-stage stochastic optimisation problem; as mentioned in the previous section, a stochastic dynamic programming approach can only solve small instances. To develop our heuristic, we proceed as follows: we approximate the original multi-stage problem as a two-stage stochastic optimisation problem with complete recourse \citep{citeulike:14541957}; and we then employ a ``receding horizon'' approach \citep{citeulike:14544211} to apply this approximation in the context of the original multi-stage setting.

Formulating the SBRP as a two-stage stochastic optimisation problem effectively means determining, at the beginning of the planning horizon, a refuelling plan comprising the optimal bowser route as well as asset replenishment quantities for all future period of the planning horizon --- note that these decisions are fixed once and for all at the beginning of the planning horizon and they do not depend on random variable realisations. Once the plan is determined, we observe random fuel consumption for all assets and periods and determine the values of recourse variables, which represent the amount of fuel spare/short for each asset at each time period. The optimal plan is the one that minimises travel costs as well as the expected total penalty cost incurred as a consequences of fuel shortages.

To model this problem we introduce the following recourse decision variables: 
\begin{itemize}
\item $[I^a_t]^-$, the expected fuel shortage for asset $a$ at time $t$, where $[I^a_0]^-=\max(-s_a,0)$;
\item $[I^a_t]^+$, the expected fuel inventory for asset $a$ at time $t$, where $[I^a_0]^+=\max(s_a,0)$;
\item $[E^a_t]$, the expected fuel quantity exceeding tank capacity for asset $a$ at time $t$.
\end{itemize}

\begin{figure}
\centering
\fbox{\parbox{14cm}{
\begin{align}
\min ~~ \sum_{t=2}^T \sum_{i=1}^N \sum_{j=1}^N T_{t-1}^{i,j}d_{i,j}+p\sum_{t=1}^T\sum_{a=1}^A [I^a_t]^-\label{cons:sto_milp_obj}
\end{align}

{\bf Subject to}
\begin{align}
&
(\ref{cons:milp1}), 
(\ref{cons:milp2}), 
(\ref{cons:milp3}), 
(\ref{cons:milp4}), 
(\ref{cons:milp5}), 
(\ref{cons:milp6}), 
(\ref{cons:milp7}), 
(\ref{cons:milp8}), 
(\ref{cons:milp9}), 
(\ref{cons:milp12})\nonumber\\
&[I^a_t]^-=
\mathcal{L}^a_{1,\ldots,t}(
s_a+
\sum_{k=1}^t Q^a_k + 
\sum_{k=1}^{t-1}[I^a_k]^- - 
\sum_{k=1}^t [E^a_k])													&t=1,\ldots,T;~a=1,\ldots,A\label{cons:milp14}\\
&[I^a_t]^+=
\widehat{\mathcal{L}}^a_{1,\ldots,t}(
s_a+
\sum_{k=1}^t Q^a_k + 
\sum_{k=1}^{t-1}[I^a_k]^- - 
\sum_{k=1}^t [E^a_k])													&t=1,\ldots,T;~a=1,\ldots,A\label{cons:milp15}\\
&[E^a_t]=\max([I^a_{t-1}]^+ + Q^a_t - c_a, 0)										&t=1,\ldots,T;~a=1,\ldots,A\label{cons:milp16}\\
\begin{split}
&T_t^{i,j},~V^i_t\in\{0,1\}\\
&Q^a_t,~B_t,~[I^a_t]^+,~[I^a_t]^-,~[E^a_t]\geq 0\\
&~0\leq S^a_t\leq f^a_t
\end{split}&\label{cons:milp17}
\end{align}
}}
\caption{An MILP model for the Bowser Routing Problem under stochastic fuel consumption}
\label{model:sto_milp_ce}
\end{figure}

The certainty equivalent MILP formulation of the problem is presented in Fig. \ref{model:sto_milp_ce}. The objective function (\ref{cons:sto_milp_obj}) now minimizes the sum of the total bowser routing cost and the expected total penalty cost incurred for fuel shortages.

Constraints 
(\ref{cons:milp1}), 
(\ref{cons:milp2}), 
(\ref{cons:milp3}), 
(\ref{cons:milp4}), 
(\ref{cons:milp5}), 
(\ref{cons:milp6}), 
(\ref{cons:milp7}), 
(\ref{cons:milp8}), 
(\ref{cons:milp9}), 
(\ref{cons:milp12}) 
are those originally discussed for the DBRP and do not change. However, asset refuelling and inventory conservation constraints (\ref{cons:milp10}) and (\ref{cons:milp11}) must be adapted in order to take into account the fact that $f^a_t$ is now a random variable. 

To compute the expected fuel shortage $[I^a_t]^-$ for asset $a$ at time $t$ as well as the expected fuel inventory $[I^a_t]^+$ for asset $a$ at time $t$, we leverage the first order loss function $\mathcal{L}^a_{1,\ldots,t}(Q)$ and the complementary first order loss function $\widehat{\mathcal{L}}^a_{1,\ldots,t}(Q)$, respectively. These are defined as follows
\begin{equation}
\widehat{\mathcal{L}}^a_{1,\ldots,t}(Q)=\sum_{k=0}^Q(Q-k)g_{1,\ldots,t}(k)
\end{equation}
\begin{equation}
\mathcal{L}^a_{1,\ldots,t}(Q)=\sum_{k=Q}^{\infty}(k-Q)g_{1,\ldots,t}(k)
\end{equation}
where $g_{1,\ldots,t}$ is the probability mass function of $f^a_1+\ldots+f^a_t$; these expressions are easily generalised to the case in which random variables are continuous \citep{citeulike:13075114}.

Our system operates under a ``lost sales'' setting; by construction, $\mathcal{L}^a_{1,\ldots,t}(Q)$ represents the sum of expected shortages observed in periods $1,\ldots,t$ if the initial fuel level is $Q$. 
\begin{lemma}
Let $\mathcal{L}^a_t(Q)$ represent expected shortages observed in period $t$, then
\[\mathcal{L}^a_t(Q)=\mathcal{L}^a_{1,\ldots,t}(Q)-\mathcal{L}^a_{1,\ldots,t-1}(Q);\] 
similarly, 
\[\widehat{\mathcal{L}}^a_t(Q)=\widehat{\mathcal{L}}^a_{1,\ldots,t}(Q)+\mathcal{L}^a_{1,\ldots,t-1}(Q).\] 
\end{lemma}
\begin{proof}
Follows from linearity of expectation.
\end{proof}

\begin{lemma}\label{lemma_5}
We can approximate $\mathcal{L}^a_t(Q)$ and $\widehat{\mathcal{L}}^a_t(Q)$ as follows
\[\mathcal{L}^a_t(Q)\approx\mathcal{L}^a_{1,\ldots,t}(Q+\mathcal{L}^a_{1,\ldots,t-1}(Q));\] 
\[\widehat{\mathcal{L}}^a_t(Q)\approx\widehat{\mathcal{L}}^a_{1,\ldots,t}(Q+\mathcal{L}^a_{1,\ldots,t-1}(Q)).\] 
\end{lemma}
\begin{proof}
These two approximations are justified by the structure of the loss function. Consider the limiting case of a complementary loss function bounded from below by a two-segment piecewise linear approximation \citep[see][p. 495 - Fig. 1]{citeulike:13075114} in which the first segment has slope 0 and the second has slope 1; the approximation here presented follows naturally by recalling that fuel consumption is defined on a nonnegative support. A similar argument applies to the loss function. In a ``lost sales'' setting shortages are reset at the end of every period; the intuition behind our approximation is to increase $Q$ by an amount equal to the sum of expected shortages observed in previous periods.
\end{proof}

\begin{lemma}\label{lemma_6}
The expected fuel quantity exceeding tank capacity for asset $a$ at time $t$ is
\[[E^a_t]=\max([I^a_{t-1}]^+ + Q^a_t - c_a, 0).\] 
\end{lemma}
\begin{proof}
Under a ``lost sales'' setting shortages are reset at the end of every period, hence the expected initial fuel level at period $t$ is $[I^a_{t-1}]^+$; $Q^a_t$ and $c_a$ are constant.
\end{proof}

Constraints (\ref{cons:milp14}) and (\ref{cons:milp15}) are obtained by applying Lemma \ref{lemma_5} and by reducing the available fuel by the sum of expected fuel quantities exceeding tank capacity for asset $a$ at times $1,\ldots,t$. Constraint (\ref{cons:milp16}) follows directly from Lemma \ref{lemma_6}. 

Constraints (\ref{cons:milp14}) and (\ref{cons:milp15}) can be easily implemented via piecewise linearization techniques presented in \citep{citeulike:13075114,rossi14} and via the \texttt{piecewise} command in IBM ILOG OPL \citep{ibm_ilog_opl}. Constraint (\ref{cons:milp16}) can be implemented using the IBM ILOG OPL \texttt{maxl} command.

Note that our approximation is similar to the one adopted in \citep{citeulike:14344640} to model expected waste in perishable inventory control.

\subsection{Numerical example}

We consider the same numerical example presented in Section \ref{sec:numerical_example_1}. However, fuel consumption in each period is now random. More specifically, fuel consumptions in different periods are independently distributed random variables that follow a Poisson distribution with mean values presented in Table \ref{tab:fuel_consumption}. We solve this instance by using the MILP approach presented in the previous section; we employ a piecewise linearisation comprising 8 segments. 

By comparing Table \ref{tab:routing_plan_small} and Table \ref{tab:routing_plan_small_stochastic} It is easy to see that the optimal routing plan did not change with respect to the deterministic problem; its cost is therefore still 494. However, by contrasting Fig. \ref{fig:refueling_plan} and Fig. \ref{fig:refueling_plan_stochastic} it is apparent that the optimal refuelling plan has changed. Moreover, since now fuel consumption is stochastic, we do expect to observe shortages. An overview on expected shortages is shown in Table \ref{tab:fuel_shortages}; the expected total amount short is 1.6831. Since the penalty cost is 100, the expected total cost predicted by the MILP model --- comprising bowser routing cost and expected fuel shortage cost is $494+100(1.6831)=662.3$. The expected total cost obtained via Monte Carlo simulation (500 replications) is 654.8 --- 95.0\% confidence interval for mean (student): (633.194,   676.406). Ideally, one may want to compare this cost against the expected total cost of the optimal plan obtained via stochastic dynamic programming. However, this instance is already too large to be solved to optimality. In Section \ref{sec:stochastic_exp} we will carry out a comprehensive numerical study on smaller instances to investigate the quality of our approximation.

\begin{figure}[h!]
\begin{minipage}[!t]{0.45\columnwidth}
\centering
\includegraphics[width=1\columnwidth]{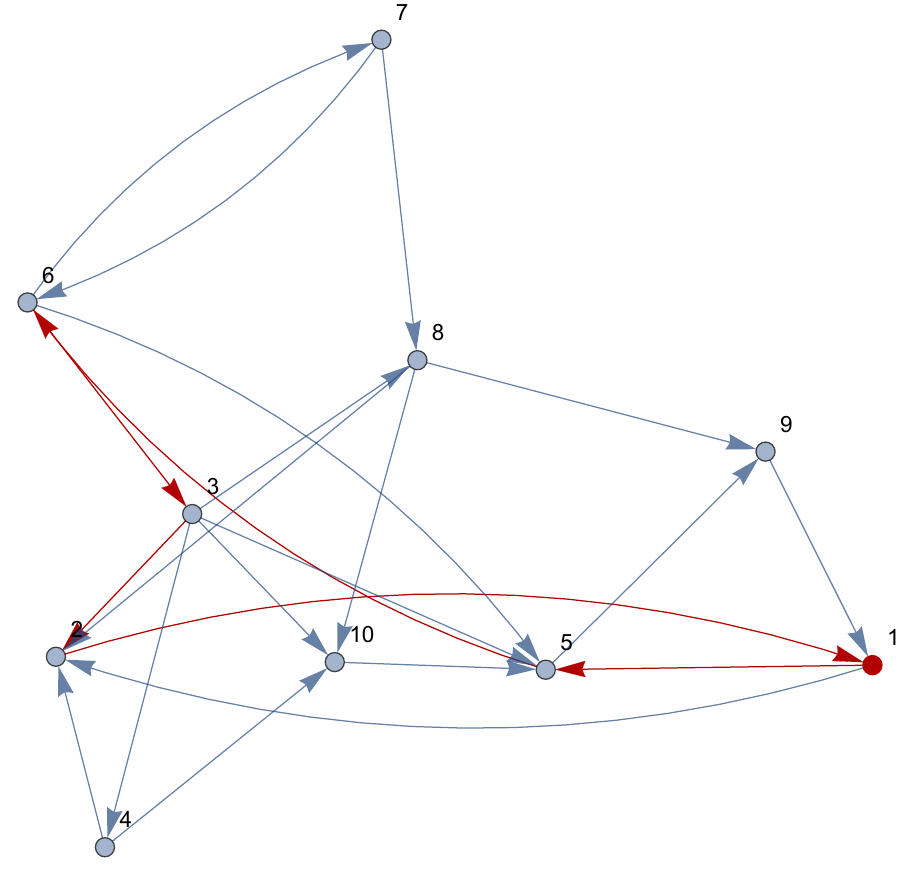}
\caption{Optimal bowser routing plan.}
\label{fig:routing_plan_stochastic}
\end{minipage}
\hspace{2em}
\begin{minipage}[!t]{0.45\columnwidth}
\centering
\begin{tabular}{llllll}
$t$	&	Transition			\\
\hline
1	&	1$\rightarrow$	1	\\
2	&	1$\rightarrow$	1	\\
3	&	1$\rightarrow$	5	\\
4	&	5$\rightarrow$	6	\\
5	&	6$\rightarrow$	3	\\
6	&	3$\rightarrow$	2	\\
7	&	2$\rightarrow$	1	\\
8	&	1$\rightarrow$	1	\\
9	&	1$\rightarrow$	1			
\end{tabular}
\caption{Optimal bowser routing plan.}
\label{tab:routing_plan_small_stochastic}
\end{minipage}
\end{figure}
\begin{figure}[h!]
\centering
\includegraphics[width=0.7\columnwidth]{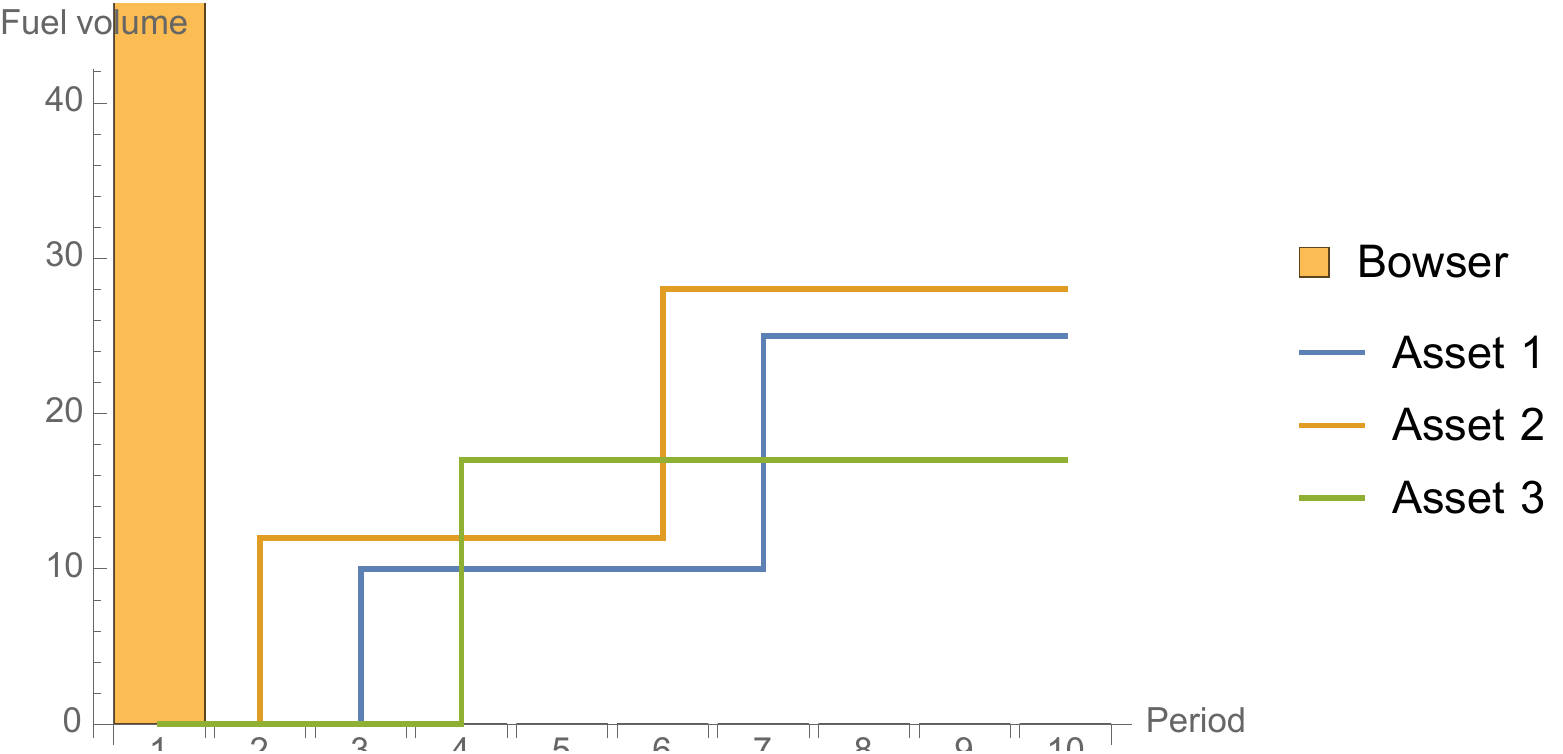}
\caption{Bowser and asset refueling plan.}
\label{fig:refueling_plan_stochastic}
\end{figure}

\begin{table}[h!]
\centering
\begin{tabular}{r|rrrrrrrrrr}
Period	& 1 & 2 & 3 & 4 & 5 & 6 & 7 & 8 & 9 & 10 \\
\hline
Asset 1	& 0.027		&0.340		&0.032		&0.038		&0.039	&0.037		&0.049	&0.051	&0.058	&0.066\\
Asset 2	& 0.027		&0.035		&0.032		&0.040		&0.043	&0.040		&0.050	&0.048	&0.055	&0.066\\
Asset 3	&0.022		&0.034		&0.153		&0.032		&0.038	&0.041		&0.043	&0.045	&0.045	&0.053
\end{tabular}
\caption{Expected fuel shortages; the expected total amount short is 1.6831.}
\label{tab:fuel_shortages}
\end{table}

An IBM ILOG OPL implementation of this example is provided in our Electronic Addendum EA2.

\section{Computational Experience}\label{sec:computational}

In this section we present our computational study for the deterministic (Section \ref{sec:deterministic_exp}) and the stochastic (Section \ref{sec:stochastic_exp}) settings. All experiments involving MILP models were performed on a MacBook Air 2.2 GHz Intel Core i7 8 GB of RAM. The optimisation environment used was IBM ILOG CPLEX Optimization Studio Version 12.6 with standard settings and a time limit of 10 minutes (600 sec). Experiments involving stochastic dynamic programming were performed on an Intel(R) Xeon(R) @ 3.5GHz with 16Gb of RAM; the library used was \texttt{jsdp}.\footnote{The code employed in our experiments is available on \url{http://gwr3n.github.io/jsdp/}.}

\subsection{Dynamic Bowser Routing Problem}\label{sec:deterministic_exp}

In this section we contrast the computational efficiency of the MILP model presented in Section \ref{sec:problem_det} to tackle the DBRP. We consider two variants of the model embedding or not valid inequalities presented. We next introduce our test bed and then present our results.

\subsubsection{Test bed}\label{sec:test_bed_det}


Fuel bowsers come in different sizes and forms, from towable tanks whose capacity ranges from 500lt to 2000lt, to tanker trucks that may reach up to 15000lt. However, due to space constraints, a large tanker truck is unlikely to be deployed within a building site. Generally, these trucks are used to carry fuel to a site cistern, and then smaller towable fuel tanks or tankers are deployed on site. For this reason we consider three levels of bowser capacity: 500lt, 1000lt, and 2000lt. Asset tank capacity vary in size depending on the type of the asset. We will consider the following assets: the JCB 540-170 telehandler, which fits a 125lt tank; the JCB 531-70 telehandler, which fits a 146lt tank; the JCB JS130, a 13 tons excavator, which fits a 235lt tank; the JCB 86C-1 mini excavator, which fits a 112lt tank. Asset initial fuel levels are uniformly distributed between 0 and 20\% of an asset tank capacity. Fuel consumption in each period of the planning horizon is randomly generated for each asset by following the distributions in Table \ref{tab:fuel_consumption}, which we obtained from our analysis of JCB LiveLink\textsuperscript{TM} asset consumption data.

\begin{figure}
\centering
\includegraphics[width=0.9\columnwidth]{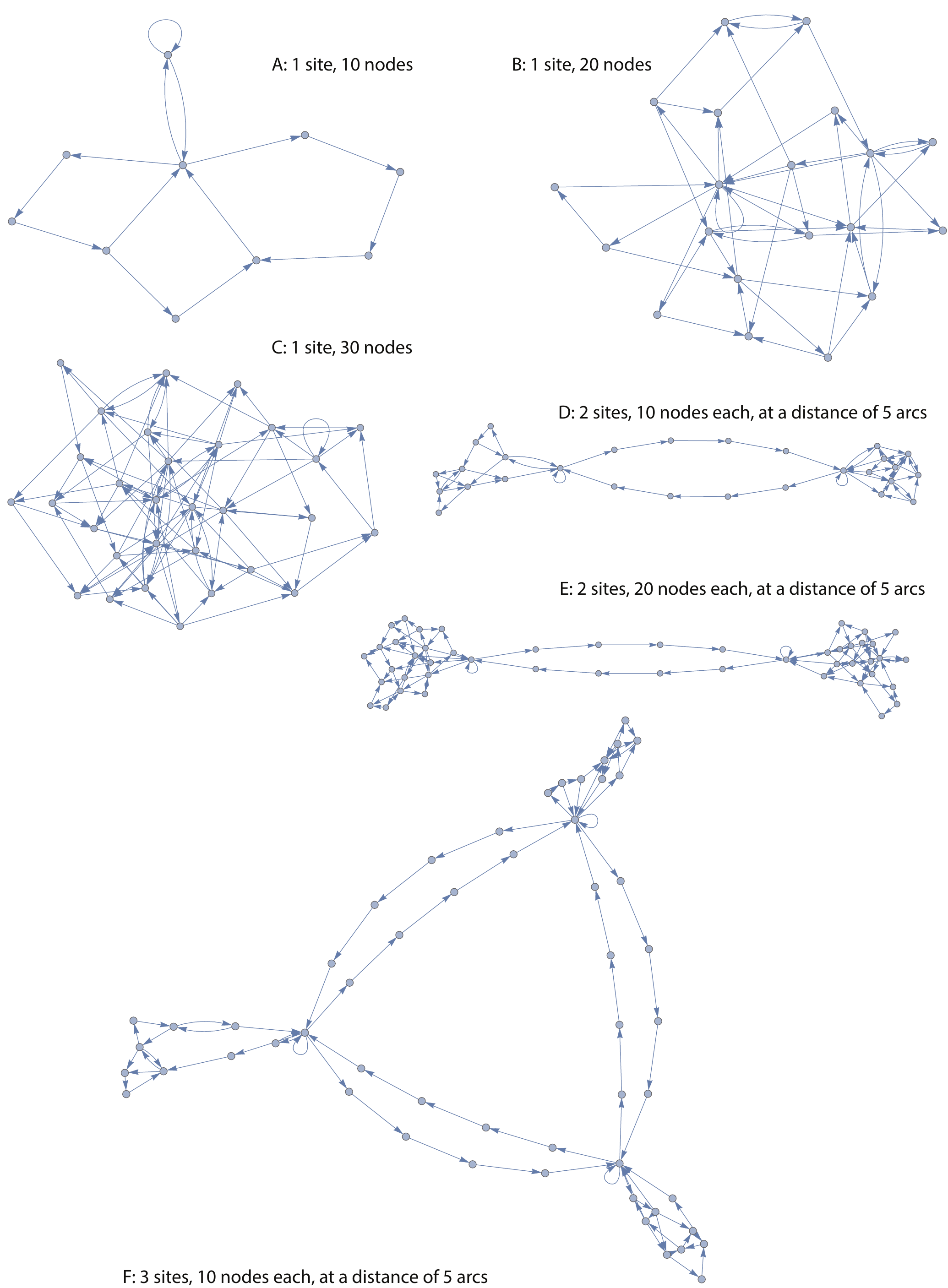}
\caption{Site topologies considered in our DBRP numerical study.}
\label{fig:topologies}
\end{figure}

We consider a test bed comprising a total of 108 instances generated as follows: the planning horizon covers $T=50$ periods. The bowser initial tank capacity takes values $c_b\in\{500,1000,2000\}$. The topologies considered in our test bed are shown in Fig. \ref{fig:topologies}; these include networks including 1, 2, and 3 sites. Site networks have been generated as Bernoulli graphs constructed starting with a complete directed graph with $n\in\{10,20,30\}$ vertices and selecting each edge independently through a Bernoulli trial with probability $p=0.1$; this process is repeated until a connected graph is obtained. Since, as discussed in Section \ref{sec:problem_det}, nodes represent adjacent accessible locations in building sites, these graphs are relatively sparse. Arc lengths are randomly generated from a normal distribution with mean 100 meters and standard deviation 20 meters. On each site we deploy a number of assets that ranges in the set $\{5,10,15\}$ --- assets on each site are randomly picked from asset types listed above. Fuel shortage penalty cost take values $p\in\{100,500\}$. The experimental design is full factorial. To ensure replicability; a complete set of IBM ILOG data files for the test bed is provided in our Electronic Addendum EA3.

\subsubsection{Results}

In what follows, we shall refer to the MILP model presented in Section \ref{sec:problem_det} as ``MP;'' the MILP model augmented with valid inequalities will be referred to as ``MPVI.''

Average key performance indicators recorded by CPLEX are presented in Table \ref{tab:cplex_kpi_det}. There are a total of 108 instances in our test bed, MP solved 84 of them to optimality --- i.e. 24 instances timed out at 600 sec), MPVI solved 90 of them to optimality --- i.e. 18 instances timed out at 600 sec. The average solution time (including timeouts) is 217 sec for MP and 122 sec for MPVI; the average time difference between MP and MPVI is charted in Fig. \ref{fig:average_time_det}. 
The average number of nodes explored is 4898 for MP and 319 for MPVI; the average node difference between MP and MPVI is charted in Fig. \ref{fig:average_node_det}. 
The average number of simplex iterations is almost 800k for MP and about 150k for MPVI; the average simplex iterations difference between MP and MPVI is charted in Fig. \ref{fig:average_iterations_det}. 
For those instances that either MP or MPVI could not solve to optimality, the average optimality gap at timeout is 45.3\% for MP and 33.4\% for MPVI; the average optimality gap difference between MP and MPVI is charted in Fig. \ref{fig:average_gap_det}. This analysis demonstrates the effectiveness of valid inequalities introduced in Section \ref{sec:problem_det}.

\begin{figure}[h!] 
\centering
\begin{minipage}[b]{0.45\linewidth}

\resizebox{\columnwidth}{!}{
\begin{tikzpicture}
\begin{axis}[
    ybar,
    ymin=0,
    xtick={0,200,400,600},
    xlabel={sec}, ylabel={instances}
]
\addplot +[
    hist={
        bins=10,
        data min=-10,
        data max=600
    },
    color=black,
    fill=black
] table [y index=0] {dataTimeDet.csv};
\end{axis}
\end{tikzpicture}
}
\caption{Time difference (sec) between DBRP and DBRP\_VI}
\label{fig:average_time_det}

\end{minipage} 
\hspace{0.05\linewidth}
\begin{minipage}[b]{0.45\linewidth}

\resizebox{\columnwidth}{!}{
\begin{tikzpicture}
\begin{axis}[
    ybar,
    ymin=0,
    xlabel={nodes}, ylabel={instances}
]
\addplot +[
    hist={
        bins=7,
        data min=0,
        data max=23000
    },
    color=black,
    fill=black
] table [y index=0] {dataNodesDet.csv};
\end{axis}
\end{tikzpicture}
}
\caption{Difference in \# explored nodes between DBRP and DBRP\_VI}
\label{fig:average_node_det}

\end{minipage} 

\vspace{2em}

\begin{minipage}[b]{0.45\linewidth}

\resizebox{\columnwidth}{!}{
\begin{tikzpicture}
\begin{axis}[
    ybar,
    ymin=0,
    xlabel={iterations}, ylabel={instances}
]
\addplot +[
    hist={
        bins=7,
        data min=0,
        data max=3000000
    },
    color=black,
    fill=black
] table [y index=0] {dataIterationsDet.csv};
\end{axis}
\end{tikzpicture}
}
\caption{Difference in \# simplex iterations between DBRP and DBRP\_VI}
\label{fig:average_iterations_det}

\end{minipage} 
\hspace{0.05\linewidth}
\begin{minipage}[b]{0.45\linewidth}

\resizebox{\columnwidth}{!}{
\begin{tikzpicture}
\begin{axis}[
    ybar,
    ymin=0,
    xlabel={gap (\%)}, ylabel={instances}
]
\addplot +[
    hist={
        bins=7,
        data min=0,
        data max=70
    },
    color=black,
    fill=black
] table [y index=0] {dataGapDet.csv};
\end{axis}
\end{tikzpicture}}
\caption{Difference between DBRP and DBRP\_VI optimality gap (\%)}
\label{fig:average_gap_det}

 \end{minipage} 
\end{figure}

\begin{table}[h!]
\centering
\begin{tabular}{lrrrrrr}
				&\multicolumn{2}{c}{time (sec)}	&\multicolumn{2}{c}{nodes}	&\multicolumn{2}{c}{simplex iterations}\\
				&MP			&MPVI		&MP			&MPVI		&MP			&MPVI	\\
\hline
Topology\\
A				&	0.32	&	1.78			&	0		&	0	&	167		&	156	\\
B				&	1.97	&	4.16			&	113		&	0	&	5697		&	627	\\
C				&	232	&	22.0			&	2208		&	46	&	503413	&	23675	\\
D				&	173	&	43.2			&	11484	&	484	&	1034887	&	76533	\\
E				&	386	&	237			&	4247		&	547	&	1076786	&	270875	\\
F				&	509	&	422			&	11334	&	837	&	2165064	&	547446	\\
\hline
Assets\\
5				&	157		&	26.2			&	5696	&	287	&	595398	&	51087	\\
10				&	135		&	114			&	2759	&	268	&	575213	&	165201	\\
15				&	358		&	224			&	6238	&	402	&	1222397	&	243367	\\
\hline
Penalty ($p$)\\
50				&	209		&	122			&	4782	&	309	&	767015	&	144513	\\
100				&	224		&	121			&	5014	&	329	&	828323	&	161924	\\
\hline													
Overall\\			
				&	217		&	122			&	4898		&	319	&	797669	&	153218	\\
\hline
\end{tabular}
\caption{Pivot table comparing average key performance indicators for MP and MPVI}
\label{tab:cplex_kpi_det}
\end{table}

\subsection{Stochastic Bowser Routing Problem}\label{sec:stochastic_exp}

In this section we investigate the effectiveness and scalability of the MILP heuristic presented in Section \ref{sec:problem_sto_milp} for the SBRP. Effectiveness will be investigated by comparing the simulated cost of the ``true'' optimal refueling plan obtained via stochastic dynamic programming against the simulated cost of the plan obtained via our MILP heuristic; in addition, we will investigate the cost performance of this heuristic under a ``receding horizon'' implementation setting \citep{citeulike:14544211}. Scalability will be investigated by tackling once more the test bed discussed in Section \ref{sec:test_bed_det} under the original random asset fuel consumption. We next introduce our test bed and then present our results.

\subsubsection{Test bed}\label{sec:test_bed_sto}

We consider a problem over a planning horizon of $T=5$ periods. The bowser capacity is set to 20 and the bowser initial tank level is 0. We consider 3 assets each of which features a tank with capacity 6; there are three initial tank level (ITL) configurations: $s^1=\{0,0,0\}$, $s^2=\{3,0,5\}$, and $s^3=\{5,5,5\}$, where $s^k_a$ is the initial tank level of asset $a$ in the $k$-th configuration. Fuel consumption of asset $k$ in period $t$ follows a Poisson distribution with mean $\lambda^k_t$ which has been truncated by fixing a maximum consumption of $7$.\footnote{Probability masses have been normalised to ensure total probability over the truncated support is 1.} We consider three possible consumption patterns (CP): in the first pattern, $\lambda^k_t=3$ for all assets and periods; in the second pattern, $\lambda^1_t=2$, $\lambda^2_t=1$, $\lambda^3_t=3$, for all $t=1,\ldots,T$; in the third pattern, $\lambda^1=\{1,2,3,4,5\}$, $\lambda^2=\{5,4,3,2,1\}$, $\lambda^3=\{3,3,1,1,2\}$, where $\lambda^k_t$ is the $t$-th entry of $\lambda^k$. We consider two possible values for the fuel stockout penalty cost $p=\{50,100\}$. There are six possible site topologies as shown in Fig. \ref{fig:topologies_stochastic} randomly generated in line with what we previously discussed; arc lengths are randomly generated from a normal distribution with mean 100 and standard deviation 20. We utilise 5 segments in our piecewise linearisation.
\begin{figure}
\centering
\includegraphics[width=0.9\columnwidth]{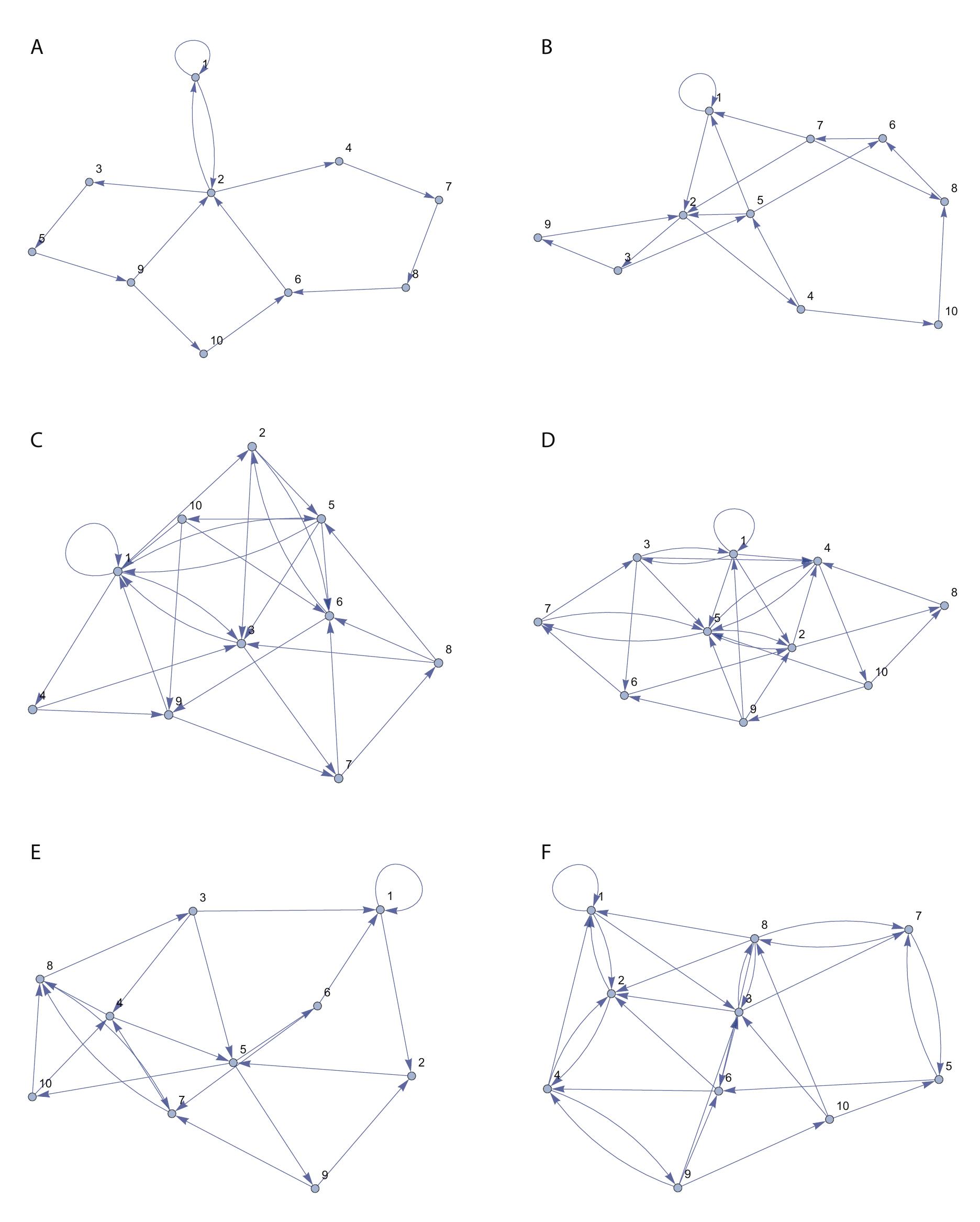}
\caption{Site topologies considered in our SBRP numerical study.}
\label{fig:topologies_stochastic}
\end{figure}
Our test bed comprises a total of 108 instances; the experimental design is full factorial.\footnote{The stochastic dynamic programming code employed in our experiments is available on \url{http://gwr3n.github.io/jsdp/}.}

\subsubsection{Results}

We first investigate the {\bf effectiveness} of the MILP heuristic presented in Section \ref{sec:problem_sto_milp} for the SBRP.  We do so by analysing two key performance indicators: the {\em linearisation gap} produced by our model, that is the difference between the expected total cost predicted by the MILP model and the expected total cost estimated by Monte Carlo simulation; and the {\em optimality gap} between the simulated cost of the refuelling plan suggested by our model and the cost of an optimal plan obtained via stochastic dynamic programming. 

In addition to investigating the cost of implementing a ``here-and-now'' plan (HN), which is a plan that fixes the bowser route and all asset replenishment quantities at the beginning of the planning horizon, we also investigate the performance of our heuristic in a ``receding horizon'' (RH) setting \citep{citeulike:14544211}. Under this setting, at any period $t=1\ldots,T$ we solve our MILP model over planning horizon $t,\ldots,T$, but we only implement time $t$ decisions --- i.e. refuelling of assets and bowser movement that occur at time $t$. After implementing these decisions, we observe fuel consumption for period $t$ and in period $t+1$ we solve again our MILP model over planning horizon $t+1,\ldots,T$ by taking into account the impact of previously observed realisations. This process continues until we have reached the end of the planning horizon. Since RH is computationally cumbersome, we limit Monte Carlo simulation to 500 replications with common random numbers \citep{citeulike:14544227}. 

Table \ref{tab:sdp_kpi_sto} report results over the test bed here investigated; the average {\em linearization gap} for our MILP model is 7.71\%, the average {\em optimality gap} is 7.59\% (HN) and 5.18\% (RH). Average solution time for the complete stochastic dynamic programming (SDP) approach is 2640 sec, while it is 0.13 sec for our MILP heuristic. These results show that our MILP approximation is effective and that its performance is enhanced under a RH strategy. 

\begin{table}[h!]
\centering
\begin{tabular}{lrrr|rr}
	&	\multicolumn{3}{c}{Gap}					&	\multicolumn{2}{c}{Time (sec)}			\\
	&	linearization	&	HN	&	RH	&	SDP	&	MILP	\\
Topology											\\
A	&	5.14	&	5.80	&	5.54	&	950	&	0.16	\\
B	&	9.20	&	13.1	&	5.51	&	2068	&	0.14	\\
C	&	6.86	&	6.69	&	4.94	&	3850	&	0.13	\\
D	&	7.00	&	5.04	&	4.49	&	2661	&	0.08	\\
E	&	8.92	&	7.53	&	5.30	&	3742	&	0.17	\\
F	&	9.16	&	7.33	&	5.29	&	2569	&	0.08	\\
\hline											
Initial tank level											\\
ITL1	&	6.55	&	4.53	&	4.02	&	1898	&	0.11	\\
ITL2	&	7.33	&	7.47	&	5.39	&	2647	&	0.12	\\
ITL3	&	9.25	&	10.78	&	6.13	&	3375	&	0.15	\\
\hline											
Consumption pattern											\\
CP1	&	5.87	&	6.01	&	4.04	&	2639	&	0.10	\\
CP2	&	8.88	&	5.96	&	2.36	&	2653	&	0.17	\\
CP3	&	8.39	&	10.82	&	9.14	&	2629	&	0.11	\\
\hline											
Penalty											\\
100	&	6.87	&	7.34	&	5.06	&	2645	&	0.12	\\
500	&	8.56	&	7.85	&	5.30	&	2635	&	0.13	\\
\hline											
Overall											\\
	&	7.71	&	7.59	&	5.18	&	2640	&	0.13	\\
\hline	
\end{tabular}
\caption{Optimality and linearization gap assessment for our SBRP MILP heuristic}
\label{tab:sdp_kpi_sto}
\end{table}

We next investigate {\bf scalability}. The MILP heuristic presented in Section \ref{sec:problem_sto_milp} for the SBRP is not as scalable as the MILP model presented in Section \ref{sec:problem_det} for the DBRP. We have run experiments on a test bed generated in line with the discussion in Section \ref{sec:test_bed_det}; however, we have now considered a shorter planning horizon comprising $T=10$ periods --- given a time bucket size of 20-30 minutes, this roughly corresponds to a three- to five-hour plan. As previously discussed, on each site we deploy a number of assets that ranges in the set $\{5,10,15\}$; these are randomly picked from asset types previously listed. For each asset type, fuel consumption in each period of the planning horizon follows the distribution in Table \ref{tab:fuel_consumption}. Once more, to ensure replicability; a complete set of IBM ILOG data files for the test bed is provided in our Electronic Addendum EA4. Results of this study are presented in Table \ref{tab:cplex_kpi_sto}. All instances could be solved within the given time limit of 600 seconds. Average solution time was 8.21 second. There is probably scope for improvement via more efficient hardware and solver setups, since we limited our study to default CPLEX setting. Future research may investigate dedicated branch-and-cut strategies similar to those discussed in \citep{tunc_et_al_2018}, which have been shown to boost scalability of MILP model based on piecewise linearisation strategies like the one here adopted.

\begin{table}[h!]
\centering
\begin{tabular}{lrrrrrr}
				&time (sec)	&nodes	&simplex iterations\\
\hline
Topology\\
A	&	2.20	&	4286	&	9530	\\
B	&	2.83	&	5465	&	12131	\\
C	&	1.44	&	1416	&	5932	\\
D	&	8.53	&	49396	&	80924	\\
E	&	1.93	&	2447	&	8698	\\
F	&	32.3	&	76087	&	132668	\\
\hline
Assets\\
5	&	2.40	&	3755	&	9938\\
10	&	16.8	&	39472	&	69368\\
15	&	5.41	&	26321	&	45636\\
\hline
Penalty ($p$)\\
50	&	14.9	&	44482	&	76571	\\
100	&	1.52	&	1884	&	6724	\\
\hline													
Overall\\			
	&	8.21	&	23183	&	41647	\\
\hline
\end{tabular}
\caption{Pivot table with average key performance indicators for our SBRP MILP heuristic}
\label{tab:cplex_kpi_sto}
\end{table}

\section{Related works}\label{sec:related_works}

The origins of the ``Truck Dispatching Problem,'' a generalization of the Traveling Salesman Problem, date back to the seminal work by \cite{citeulike:6888144}. Since the early days, a sizeable literature developed on the so-called Vehicle Routing Problem (VRP), whose aim is to dispatch a fleet of vehicles on a given network to serve a set of customers while meeting a number of constraints. A comprehensive discussion on the VRP can be found in \citep{citeulike:12046141,citeulike:14131446}. 

Equally fundamental in production economics is the literature on inventory control. Pioneering works in this area were carried out by \cite{citeulike:14131448}, who introduced the concept of ``economic order quantity,'' and \cite{citeulike:14131449}, who discussed the first lot-sizing algorithm for a finite-horizon inventory system subject to dynamic demand. Lot-sizing models are surveyed in \citep{citeulike:14131467,citeulike:14131470}; recent developments in stochastic lot sizing are surveyed in \citep{citeulike:12394813}.

The field of Inventory Routing (IR), whose origins can be traced back to the work of \cite{citeulike:14131451}, encompasses problems which combine vehicle routing and inventory management decisions. In IR optimisation is delegated to a central entity that jointly optimises all decisions. Recent surveys in the area include \citep{citeulike:12261199,citeulike:6240342,citeulike:14131452}. The first exact approach to the Inventory Routing Problem (IRP) was proposed by \cite{citeulike:14131453}, who considered the single-vehicle case. Computationally efficient approaches to this problem were discussed in \citep{citeulike:14131455,citeulike:10228296}. Approaches for the multi-vehicle case include \citep{citeulike:10647329,citeulike:11425703,citeulike:13938310}. 
Stochastic IRP problems, in which customer demand is modeled as a random variable, include the seminal work by \cite{10.2307/170651}, and a number of more recent contributions \citep{10.2307/3088529,citeulike:12252336,citeulike:12252318,citeulike:12252334,citeulike:6399125,citeulike:12252325,citeulike:7614853,citeulike:12252352}.

The DBRP, discussed in this work, presents a number of similarities and differences with the classical IRP and its existing variants presented in the literature. Like the classical IRP the DBRP features a warehouse (the fuel cistern), a truck shipping an ``item'' (the fuel), and a number of recipients (the assets) for such item. Like retailers in the IRP, assets hold inventory (the fuel) that is depleted over time. The key difference from existing works in the IRP literature is the fact that assets are free to move in the network that represents our construction site layout; this means that an asset location may change over time and that an asset will move across the network over time together with its inventory. 

Works on Probabilistic Traveling Salesman Problem \citep{citeulike:7695005} and Stochastic Vehicle Routing Problem \citep{citeulike:4008511} have considered situations in which a customer may be present at a given node of the network with a prescribed probability, but we are not aware of works investigating the case in which the very same customer, which holds inventory that may need to be replenished, moves from one node to another over time in a deterministic or stochastic fashion.

A number of other works in the literature share similarities with our study, although in different application domains. We hereby provide a survey of what we believe are the most relevant application domains; for each domain, we try to stress the key differences from our problem.

\citep{citeulike:14147392,citeulike:4008511} investigated arc routing problems motivated by winter gritting applications where the ``timing'' of each intervention is crucial and service cost increases if intervention is carried out outside the prescribed time window. Our model also feature similar ``timing'' constraint related to asset fuel stockout, but in road gritting, as in classical IRP, assets that are replenished (i.e. gritted roads) do not move from one period to the next.

\citep{citeulike:14147415,citeulike:14147417,citeulike:14147413} represent a sample of works in a stream of literature that gained substantial momentum in recent times: electric vehicle refueling. In electric vehicle refueling assets do move over time from one node of the road network to the next together with their battery, which can be seen as energy inventory. However, refueling stations have a fixed location and the problem typically addresses a strategic decision whose aim is to determine where to position refueling stations. 

Rebalancing in bike sharing \citep{citeulike:14147432,citeulike:14147430} is another problem that attracted considerable interest in recent time. In this problem, the aim is to utilise one or more trucks in order to regularly rebalance the number of bikes located at self-service stations. The ``pick-up and delivery'' nature of this problem clearly differentiates it from our problem.

Moving our attention to robotics, recently \citep{citeulike:14544206,citeulike:14544208,citeulike:14544091} investigated stochastic collection and replenishment of agents motivated by use cases in mining and agricultural settings in which a replenishment agent transports a resource between a centralised replenishment point to agents using the resource in the field. They employ Gaussian approximations to quickly calculate the risk-weighted cost of a schedule; a branch and bound search then exploits these predictions to minimise the downtime of the agents. Previous works in this area mainly focused on scheduling the actions of the dedicated replenishment agent from a short-term and deterministic angle. Scenario 2 \citep[see][p. 59]{citeulike:14544091} present similarities to our setup; this emphasises the practical relevance of our study. However, our modeling and solution framework has the advantage of relying solely on MILP modeling and not on ad-hoc algorithms to predict future asset resource levels. Moreover, the discussion in \cite{citeulike:14544091} assumes that uncertain parameters and variables are normally distributed, while our approach based on piecewise linearisation of loss functions does not require this assumption and can accommodate any distribution.

Finally, our problem can be seen as a simplified version of the Aerial Fleet Refueling Problem (AFRP) \citep{citeulike:14147444} in which the aim is to schedule deployment of tankers and receiver aircraft, located at diverse geographical locations, in support of immediate and anticipated military operations. In contrast to the DBRP, the AFRP features a complex multicriteria and hierarchical objective function that does not simply minimize the distance covered by the tanker aircraft. Furthermore, this problem features a substantial number of additional constraints related to safety of crew and aircrafts. A number of more recent works investigate the Aerial Refueling Scheduling Problem (ARSP) \citep{citeulike:9840687}, whose aim is to determine the refueling completion times for fighter aircrafts on multiple tankers in order to minimize the total weighted tardiness. The ARSP however does not consider the topology of the theatre of operations.

\section{Conclusion}\label{sec:con}

In this work we focussed on efficient refuelling of assets across construction sites. We introduced the Bowser Routing Problem and discussed its deterministic (DBRP) and stochastic (SBRP) variants. For each of these variants, we developed mathematical programming models that, by leveraging data supplied by different assets, schedule refuelling operations by balancing the cost of dispatching a bowser truck with that of incurring fuel shortages. To enhance computational performances, we discussed valid inequalities for our models. We carried out  a comprehensive set of experiments on a testbed designed around data derived from our experience at Crossrail sites. In a deterministic setting, this study shows that our mathematical programming models can tackle instances of realistic size in reasonable time --- generally ranging from seconds to minutes; valid inequalities on average halve computational time and dramatically reduce explored nodes and simplex iterations required to reach an optimal solution. In the stochastic case, our mathematical programming heuristic is effective and produces tight linearisation (approx. 7\%) as well as optimality (approx. 5\%) gaps when contrasted against optimal solutions obtained via stochastic dynamic programming for small instances. Computationally, our heuristic provides reasonable performances. An interesting direction for future research is to investigate more effective reformulations for this latter MILP heuristic.

\section*{Acknowledgements}

This work has been developed in the context of project ReCCEL, which was supported by Innovate UK grant number 55202 - 418137.

\bibliographystyle{plainnat}
\bibliography{publications}

\newpage

\section*{Appendix I}
In Fig. \ref{fig:aemp_xml} we provide a sample output obtained from the JCB LiveLink system, which relies on the AEMP standard v1.2.
\begin{figure}[h!]
\scriptsize
\lstset{language=XML}
\begin{lstlisting}
<?xml version="1.0" encoding="utf-8"?>
<Fleet version="0" snapshotTime="2016-06-18T16:02:32.2804358Z" 
xmlns="http://schemas.aemp.org/fleet" 
xmlns:xsd="http://www.w3.org/2001/XMLSchema" 
xmlns:xsi="http://www.w3.org/2001/XMLSchema-instance">
<Equipment>
<EquipmentHeader>
<UnitInstallDateTime>2015-04-20T12:38:25.07</UnitInstallDateTime>
<Make>JCB</Make>
<Model>JS130</Model>
<EquipmentID>Axxxxxx</EquipmentID>
<SerialNumber>xxxxxxx</SerialNumber>
</EquipmentHeader>
<Location datetime="2016-06-18T11:37:59.807">
<Latitude>52.7990309</Latitude>
<Longitude>-2.2744561</Longitude>
</Location>
<CumulativeOperatingHours datetime="2016-06-18T11:37:58">
<Hour>P28DT7H</Hour>
</CumulativeOperatingHours>
<FuelUsed datetime="2016-06-18T11:37:58">
<FuelUnits>liter</FuelUnits>
<FuelConsumed>4902</FuelConsumed>
</FuelUsed>
<Distance datetime="2016-06-18T11:37:58">
<OdometerUnits>kilometer</OdometerUnits>
<Odometer>0</Odometer>
</Distance>
</Equipment>
</Fleet>
\end{lstlisting}
\caption{Sample AEMP v1.2 XML output}
\label{fig:aemp_xml}
\end{figure}

\newpage

\section*{Appendix II}

In Fig. \ref{fig:AssetsFuelConsumption_1} and Fig. \ref{fig:AssetsFuelConsumption_2} we report fuel consumption heatmaps and average hourly consumption for a range of assets deployed on Crossrail sites between the 1st of June and the 30th of June 2016; for each asset, we highlight the representative time window that we employed in our distribution fitting exercise (Table \ref{tab:fuel_consumption}); note that assets are presented in the same order followed in Table \ref{tab:fuel_consumption}, the first asset in Table \ref{tab:fuel_consumption} is not included here, since its fuel consumption heatmap and average hourly fuel consumption graph have been presented in the main text.

\begin{figure}[h!]
\centering
\includegraphics[width=1\columnwidth]{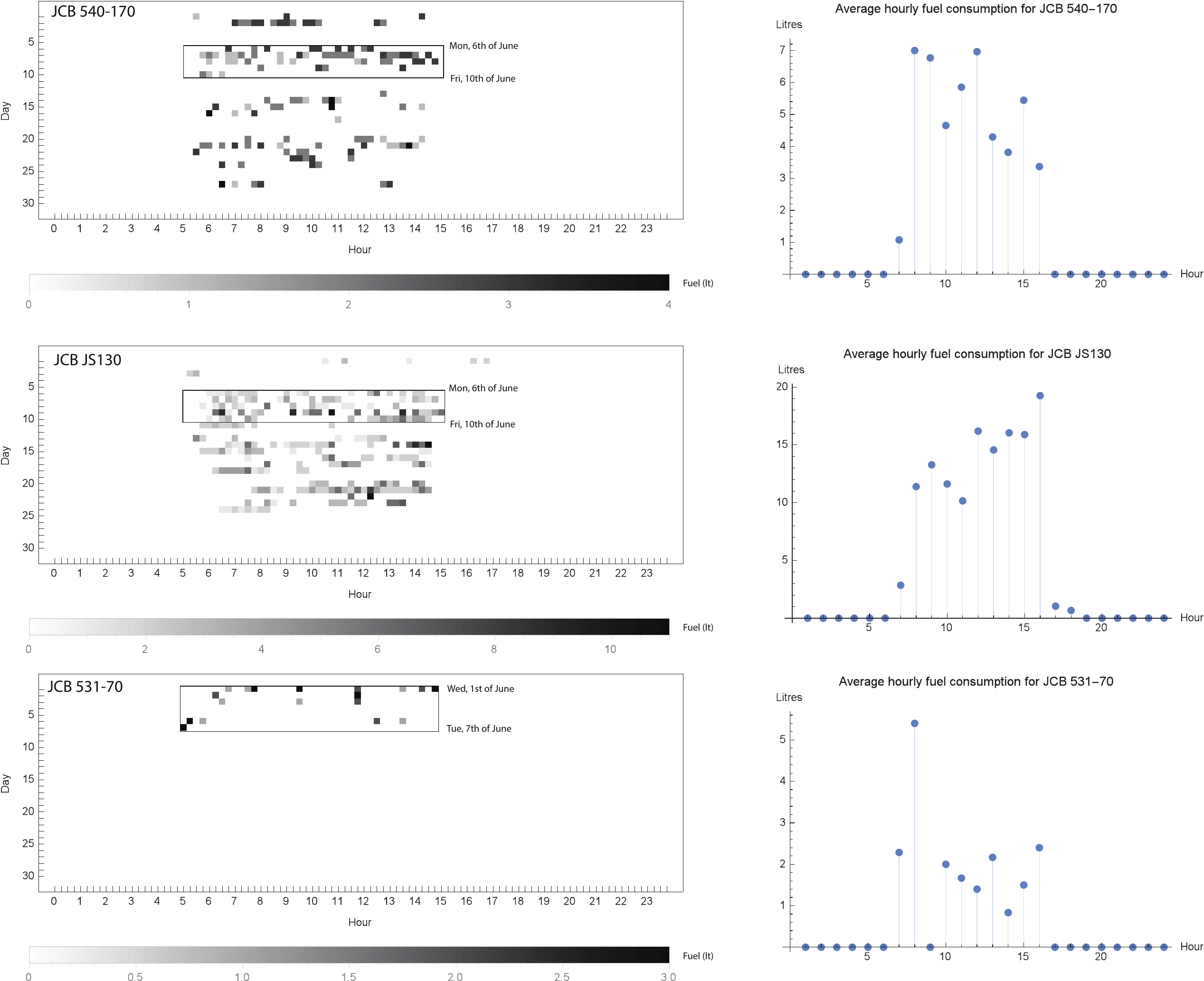}
\caption{Fuel consumption heatmaps and average hourly fuel consumption for a range of assets deployed on Crossrail sites between the 1st of June and the 30th of June 2016; for each asset, we indicate the representative time window that we employed for distribution fitting.}
\label{fig:AssetsFuelConsumption_1}
\end{figure}

\begin{figure}[p]
\centering
\includegraphics[width=1\columnwidth]{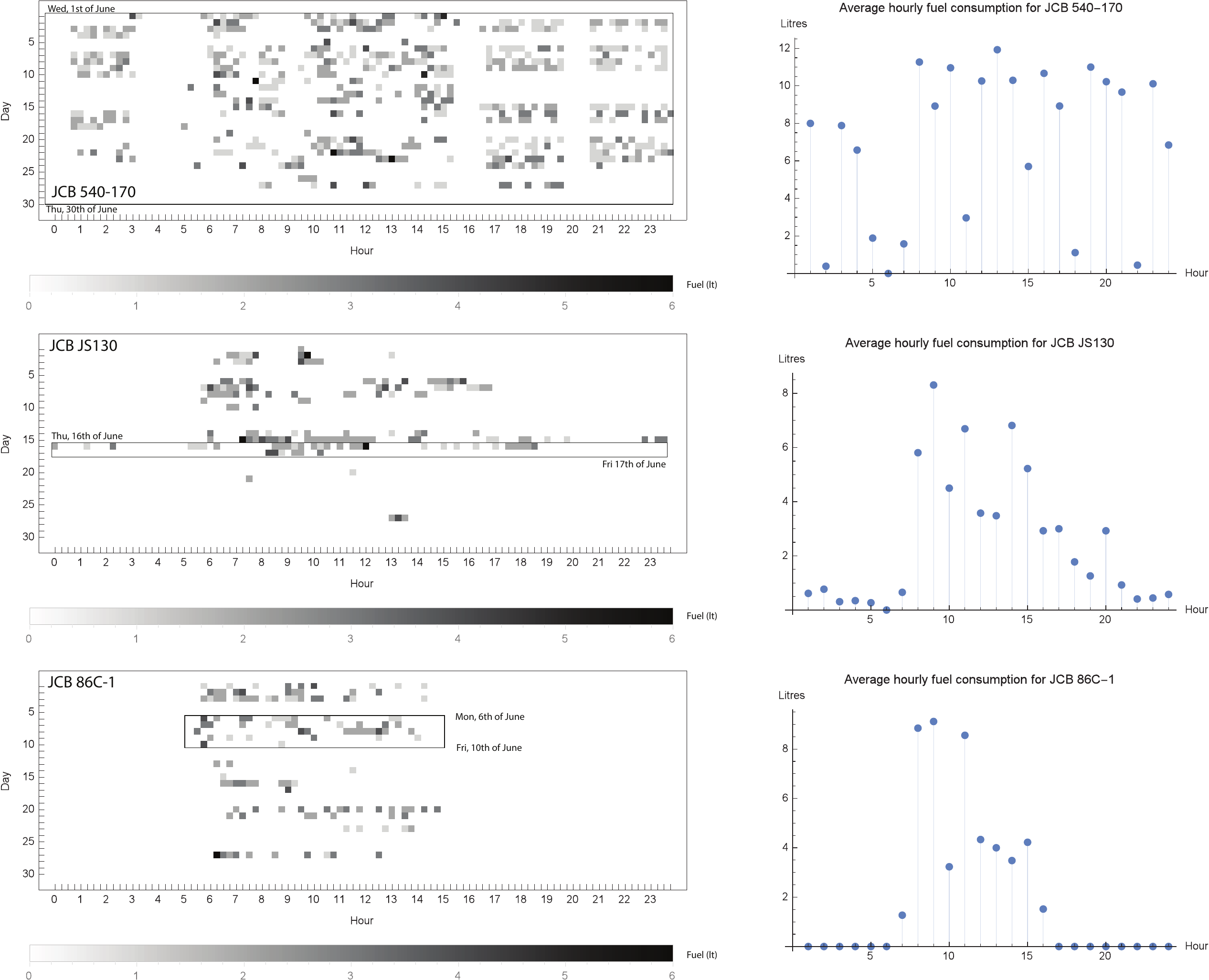}
\caption{Fuel consumption heatmaps and average hourly fuel consumption for a range of assets deployed on Crossrail sites between the 1st of June and the 30th of June 2016; for each asset, we indicate the representative time window that we employed for distribution fitting.}
\label{fig:AssetsFuelConsumption_2}
\end{figure}

\newpage

\section*{Appendix III}

Symbols used in the manuscript are listed in Table \ref{tab:symbols}.

\begin{table}[h!]
\begin{tabular}{ll}
\bf{Parameters}\\
\hline
$T$		&number of time periods;\\
$A$		&number of assets;\\
$N$		&number of nodes in the site network (i.e. $N=|V|$);\\
$d_{i,j}$	&distance between node $i$ and node $j$ in the site\\
		&network, if $i=j$, $d_{ij}=0$;\\
$\delta_{i,j}$&a binary variable that is set to one iif it \\
		&is possible to travel from node $i$ to node $j$ in one time period;\\
$l^a_{t,i}$	&a binary variable that is set to one iif asset $a$ \\
		&is at node $i$ during time period $t\in T$; i.e iif $l^a_t=i$;\\
$f^a_t$	&fuel consumption of asset $a$ in time period $t\in T$\\
$F$		&total fuel consumption for all assets across all time periods;\\		
$c_a$	&tank capacity of asset $a$;\\
$s_a$	&initial tank level of asset $a$;\\
$c_b$	&bowser tank capacity;\\
$s_b$	&initial bowser tank level;\\
$p$		&penalty cost per litre of fuel short;\\
\\
\bf{Decision variables}\\
\hline
$V^i_t$	&a binary variable that is set to one iif, at time $t$, \\
		&the bowser is at node $i$;\\
$T_t^{i,j}$	&an auxiliary binary variable that is set to one iif \\
		&the bowser transits from node $i$ to node $j$ by the end \\
		&of period $t$;	\\
$Q^a_t$	&the quantity of fuel delivered to asset $a$ at time $t$;\\		
$S^a_t$	&the fuel shortage for asset $a$ at time $t$;\\		
$B_t$	&the quantity of fuel transferred from the site cistern to the\\
		&bowser at time $t$;\\	
$[I^a_t]^-$&the expected fuel shortage for asset $a$ at time $t$,\\
		&where $\mathrm{E}[I^a_0]^-=\max(-s_a,0)$;\\
$[I^a_t]^+$&the expected fuel inventory for asset $a$ at time $t$, \\
		&where $\mathrm{E}[I^a_0]^+=\max(s_a,0)$;\\
$[E^a_t]$&the expected fuel quantity exceeding tank capacity \\
		&for asset $a$ at time $t$.\\		
\hline
\end{tabular}
\caption{Summary table of symbols used in this work.}
\label{tab:symbols}
\end{table}

\end{document}